\documentclass[pdflatex,sn-mathphys-num]{sn-jnl}

\usepackage{graphicx}%
\usepackage{multirow}%
\usepackage{amssymb}
\usepackage{amsmath}
\usepackage{amsfonts}
\usepackage{mathrsfs}
\usepackage{mathtools}
\usepackage[title]{appendix}%
\usepackage{xcolor}%
\usepackage{textcomp}%
\usepackage{manyfoot}%
\usepackage{booktabs}%
\usepackage{listings}%
\usepackage{enumitem}
\usepackage{centernot}
\usepackage{tikz}
\usetikzlibrary{arrows.meta,positioning,fit,calc,decorations.pathreplacing,shapes.multipart,backgrounds}
\pgfdeclarelayer{background}
\pgfsetlayers{background,main}
\usepackage{cleveref}
\usepackage{graphicx}
\usepackage[ruled,vlined]{algorithm2e}
\DontPrintSemicolon                
\SetKwInput{KwData}{Input}
\SetKwInput{KwResult}{Output}
\SetKwInput{KwInit}{Initialization}

\theoremstyle{thmstyleone}
\newtheorem{theorem}{Theorem}[section]
\newtheorem{proposition}[theorem]{Proposition}
\newtheorem{lemma}[theorem]{Lemma}
\newtheorem{corollary}[theorem]{Corollary}
\theoremstyle{definition}
\newtheorem{definition}[theorem]{Definition}
\newtheorem{assumption}[theorem]{Assumption}
\theoremstyle{remark}

\newcommand{\Om}{\Omega}
\newcommand{\R}{\mathbb{R}}
\newcommand{\C}{\mathbb{C}}
\newcommand{\Tcal}{\mathcal{T}}
\newcommand{\Kcal}{\mathcal{K}}
\newcommand{\Ev}{\mathcal{E}}
\newcommand{\Rcal}{\mathcal{R}}
\newcommand{\Ccal}{\mathcal{C}}
\newcommand{\Lcal}{\mathcal{L}}
\newcommand{\Mm}{\mathcal{M}}
\newcommand{\Gcal}{\mathcal{G}}

\newcommand{\Uad}{\mathcal{U}}
\newcommand{\Bt}{\mathcal{B}_T}
\newcommand{\Btl}{\mathcal{B}_T^{\mathrm{Lip}}}
\newcommand{\Lip}{\operatorname{Lip}}
\newcommand{\norm}[1]{\left\lVert#1\right\rVert}
\newcommand{\abs}[1]{\left\lvert#1\right\rvert}

\numberwithin{equation}{section}

\begin{document}
 
\title{Endogenous Feedback in Size-Structured Transport Equations}  

\author[1]{\fnm{Jiguang} \sur{Yu}}\email{jyu678@bu.edu}
\equalcont{These authors contributed equally to this work as co-first authors.}

\author*[2]{\fnm{Louis Shuo} \sur{Wang}}\email{wang.s41@northeastern.edu}

\affil[1]{\orgdiv{College of Engineering},
  \orgname{Boston University},
  \orgaddress{\city{Boston}, \postcode{02215}, \state{MA}, \country{USA}}}

\affil[2]{\orgdiv{Department of Mathematics},
  \orgname{Northeastern University},
  \orgaddress{\city{Boston}, \postcode{02115}, \state{MA}, \country{USA}}}

\abstract{
We study a nonlinear size-structured transport equation where the endogenous scalar output $E(t)=\int_{l_0}^{l_m}\chi(l)x(t,l)\,dl$ feeds back into velocity and mortality. This principal-coefficient feedback precludes a semilinear perturbation framework. Freezing the feedback path yields a non-autonomous linear evolution, reducing the closed-loop problem to a scalar Volterra fixed point $E=\mathcal K(E)$. Mass balance provides an intrinsic feedback interval, while a Bielecki-norm contraction ensures unique nonnegative weak solutions.

Stationary equilibria satisfy a scalar closure equation $E=\Phi(E)$. We prove uniqueness below the sharp margin $1-\Phi'(E)>0$ and identify $\Phi'(E)=1$ as a nondegenerate fold threshold. Linearization yields a finite-memory renewal equation with characteristic equation $\mathcal E(\lambda)=1$, whose root set determines the feedback spectrum and stability. Finally, the stationary harvesting adjoint reduces to a rank-one perturbation formula. At zero discount, we establish the identity $\mathcal E(0)=\Phi'(E^*)=B(0)$, linking closure resonance, spectral crossing, and adjoint loop gain.
}



\maketitle

\section{Introduction}
\label{sec:introduction}

Let $\Omega=[l_0,l_m]\Subset(0,\infty)$ with $0<l_0<l_m<\infty$, and let $x(t,l)\ge0$ denote the density of individuals structured by size $l$. In this paper, we study controlled transport equations of the form
\begin{equation}
\label{eq:intro-state}
\partial_t x(t,l)
+
\partial_l\!\big(g(E(t),l)x(t,l)\big)
=
-\big(\mu(E(t),l)+u(t,l)\big)x(t,l),
\qquad l\in(l_0,l_m),
\end{equation}
where the scalar feedback variable is the endogenous output
\begin{equation}
\label{eq:intro-output}
E(t)=\int_{l_0}^{l_m}\chi(l)x(t,l)\,dl .
\end{equation}
Thus, the model exhibits a closed feedback architecture given by $x \longmapsto E=\langle\chi,x\rangle \longmapsto \big(g(E,\cdot),\mu(E,\cdot)\big) \longmapsto x$. For a prescribed path $E$, equation~\eqref{eq:intro-state} is a linear non-autonomous transport equation. In the closed-loop problem, however, $E$ is not prescribed; it must satisfy the causal fixed-point relation
\begin{equation}
\label{eq:intro-fixed-point}
E=\mathcal K(E),
\qquad
\mathcal K(E)(t)
:=
\int_{l_0}^{l_m}
\chi(l)\,\mathcal T_E(x_0,p,u)(t,l)\,dl ,
\end{equation}
where $\mathcal T_E$ denotes the frozen-path transport solution operator. The resulting problem is therefore not merely a semilinear perturbation of a fixed generator. Because the feedback enters the principal transport coefficient $g(E(t),l)$, the transport geometry itself depends on the state through the output $E(t)$. This places the model within the tradition of nonlinear age- and size-structured population dynamics; see, for example, \cite{diekmann1998formulation,gurtin1974non,iannelli1995mathematical,metz2014dynamics,perthame2007transport,thieme2018mathematics,webb1985theory,liu2026ftu}.

Recruitment in the system is imposed through the inflow flux boundary condition and initial state
\begin{equation}
\label{eq:intro-flux-bc}
g(E(t),l_0)x(t,l_0)=p(t),
\qquad
x(0,l)=x_0(l)\ge0 .
\end{equation}
Integrating \eqref{eq:intro-state} over $\Omega$ gives, at the formal level,
\[
\frac{d}{dt}\int_{l_0}^{l_m}x(t,l)\,dl
=
p(t)
-
g(E(t),l_m)x(t,l_m)
-
\int_{l_0}^{l_m}
(\mu(E(t),l)+u(t,l))x(t,l)\,dl .
\]
Because the outflow term and the removal term are non-negative, one obtains the key a priori bound $\|x(t,\cdot)\|_{L^1(\Omega)} \le \|x_0\|_{L^1(\Omega)} + \int_0^t p(s)\,ds$. Consequently,
\begin{equation}
\label{eq:intro-MT}
0\le E(t)\le
\|\chi\|_{L^\infty(\Omega)}
\left(
\|x_0\|_{L^1(\Omega)}
+
\int_0^T p(s)\,ds
\right)
=:M_T .
\end{equation}
Thus, the admissible feedback interval $[0,M_T]$ is naturally derived directly from the intrinsic system dynamics rather than being imposed as an external compactness assumption.

For a stationary control $u=u(l)$ and constant input $p>0$, the frozen equilibrium at a fixed feedback level $E$ is given by
\[
x_E(l)=
\frac{p}{g(E,l)}
\exp\!\left[
-\int_{l_0}^{l}
\frac{\mu(E,\xi)+u(\xi)}{g(E,\xi)}\,d\xi
\right].
\]
The stationary closed-loop equation then reads $E=\Phi(E)$, where $\Phi(E):=\int_{l_0}^{l_m}\chi(l)x_E(l)\,dl$. Defining $H(E):=E-\Phi(E)$, the sharp local degeneracy condition is $H'(E^*)=0$, which is equivalent to $\Phi'(E^*)=1$. This condition marks the closure-resonance threshold. In a one-parameter family $H(E,\eta)=E-\Phi(E;\eta)$, the non-degenerate fold conditions are characterized by $H(E_0,\eta_0)=0$, $H_E(E_0,\eta_0)=0$, $H_{EE}(E_0,\eta_0)\ne0$, and $H_\eta(E_0,\eta_0)\ne0$.

The same scalar quantity governs the stationary adjoint feedback structure; for optimal harvesting of structured populations and the associated adjoint systems, see \cite{anita2013analysis,wang2025analysis}. For the discounted yield
\[
J_r(u)=
\int_0^\infty e^{-rt}
\int_{l_0}^{l_m}\pi(l)u(t,l)x(t,l)\,dl\,dt ,
\]
the frozen current-value adjoint has the form
\[
\mathcal{L}_r^*\lambda
=
-g^*\lambda'
+(r+\mu^*+u)\lambda,
\qquad
\lambda(l_m)=0 .
\]
The scalar feedback produces a rank-one perturbation $\big(\mathcal{L}_r^*-\chi\langle\sigma,\cdot\rangle\big)\lambda = q_{\rm red}$, where the co-load $\sigma$ contains the boundary Dirac correction induced by the flux constraint \eqref{eq:intro-flux-bc}. Hence,
\[
\lambda
=
R_r q_{\rm red}
+
\frac{\langle\sigma,R_r q_{\rm red}\rangle}
     {1-\langle\sigma,R_r\chi\rangle}
R_r\chi,
\qquad
R_r:=(\mathcal{L}_r^*)^{-1}.
\]
At $r=0$, the forward sensitivity, adjoint loop gain, and linearized characteristic function satisfy the headline identity $\mathcal{E}(0)=\Phi'(E^*)=B(0)$. Consequently,
\[
\Phi'(E^*)=1
\qquad\Longleftrightarrow\qquad
\mathcal{E}(0)=1
\qquad\Longleftrightarrow\qquad
B(0)=1 .
\]
Thus, the scalar margin $1-\Phi'(E^*)$ is simultaneously the stationary closure margin, the zero-eigenvalue margin of the linearized renewal equation, and the zero-discount rank-one adjoint margin. Away from the zero crossing, linearized stability is determined by the full root set $\mathcal{E}(\lambda)=1$, rather than by $\Phi'(E^*)$ alone.

The organization of this paper is as follows. Section~\ref{sec:model} introduces the model, foundational assumptions, weak formulation, frozen-path operator, stationary closure, and the stationary harvesting framework. Section~\ref{sec:main-results} presents the main theoretical results. The detailed mathematical proofs are provided in Section~\ref{sec:proofs}. Section~\ref{sec:conclusion} concludes the paper. Finally, the appendix contains essential auxiliary transport estimates alongside the nonlinear stability verification.

\section{Literature Review}
The study of physiologically structured population models, alongside their simplification into delay or integral equations, is a well-established field of research. Classical works frequently utilize the transport (McKendrick--von Foerster) equations to describe age- and size-structured dynamics, analyzing them via semigroup theory and the method of characteristics~\cite{gurtin1974non,liang2025global,metz2014dynamics,wang2026damage,webb1985theory,yu2026microscopic,yu2026rigorous}. Comprehensive textbook treatments covering the associated mathematical modeling, analysis, and bifurcation theory can be found in~\cite{cushing1998introduction,iannelli1995mathematical,iannelli2017basic,metz2014dynamics,murray2007mathematical,perthame2007transport,liu2025bidirectional,de1997gentle,thieme2018mathematics}. Methodical approaches to reformulating nonlinear structured systems featuring feedback---typically mediated by one or multiple scalar environmental interaction variables---have been extensively explored in~\cite{akimenko2018two,barril2022formulation,barril2024hierarchical,diekmann2001formulation,diekmann2003steady,diekmann1998formulation,diekmann2020finite,diekmann2020models,liu2026computational}. This prior work directly inspires our approach of defining the closed-loop system as a fixed-point problem with respect to the environmental input $E$. The theoretical framework utilizing abstract semilinear equations and integrated semigroups for these structured problems has been rigorously established in~\cite{ducrot2021integrated,kang2022age,magal2018theory,wang2025analysis1}. Furthermore, investigations into the local stability and bifurcation behavior of prototypical feedback models have been conducted in~\cite{diekmann2010daphnia,scarabel2021numerical,wang2026algebraic}. Results concerning the existence and stability of solutions in transport models incorporating feedback, primarily utilizing characteristic curves and fixed-point theorems, are presented in~\cite{bartlomiejczyk2015existence,calsina1995model,yu2026pattern} and similar literature. From a dynamical systems perspective, it is customary within the delay and renewal equation framework to reduce linearized structured models to Volterra convolution (renewal) equations and to determine stability via characteristic equations~\cite{diekmann2012delay,hale2013introduction,wang2026breakdown}. The necessary theoretical foundations for convolution resolvents and spectrum-determined growth are provided by the theory of Volterra integral equations involving measure kernels~\cite{gripenberg1990volterra,yu2026beyond} as well as standard semigroup theory~\cite{engel2000one,wang2026elliptic}. The one-dimensional transport estimates and $L^1$-contraction properties employed in our work are conventional for linear conservation laws dealing with bounded variation ($BV$) data~\cite{bressan2000hyperbolic,bressan2023remark,diperna1989ordinary,wang2025multi}. Additionally, we consistently apply the $BV$ calculus techniques detailed in~\cite{ambrosio2000functions,gao2022rolling}. Lastly, there exists a vast body of literature dedicated to optimal harvesting and control problems within size- and age-structured dynamics, which includes the derivation of Pontryagin maximum principles and associated optimality systems~\cite{anita2013analysis,brokate1985pontryagin,feichtinger2003optimality,cai2026optimal,filho2025mathematical,kato2024measure,kumar2023stability,ni2023optimal,wu2025age,wu2025analysis}. An additional segment of the relevant literature concentrates on selective harvesting policies within structured populations, aiming to determine the optimal demographic segments to harvest in order to maximize long-term or discounted economic profits \citep{easterling2000size,fadlovich2025selectivity,stubberud2019effects,liang2026separation}. While this current manuscript focuses on establishing the well-posedness and stability properties of the closed-loop system, the corresponding stationary optimal control problem is reserved for future investigations.

Within this context, the primary contributions of our study are threefold. First, our utilization of the inflow flux framework renders the associated feedback interval $[0,M_T]$ intrinsic: rather than being introduced as an a priori boundedness restriction, it naturally emerges from a mass conservation property (Equation~\ref{eq:model-MT}). Second, we demonstrate that the stationary closure slope permits a precise decomposition, $\Phi'=\Rcal-\Ccal$, separating into residence-time amplification and cumulative survival attenuation (Proposition~\ref{prop:main-decomp}), thereby clarifying the algebraic sign of this slope. Third, and of greatest significance, we identify a unique scalar margin, $1-\Phi'(E^*)$, and establish that this identical quantity concurrently controls the breakdown of local closure invertibility (i.e., a stationary fold) and the zero characteristic root of the linearized closed-loop system. This is formalized via the equivalence $\Ev(0)=\Phi'(E^*)$ concerning the characteristic function $\Ev$ associated with a finite-memory scalar renewal equation. This phenomenon of closure-resonance, which intrinsically links a stationary bifurcation to a dynamic spectral boundary using a single, explicitly computable scalar, serves as the fundamental organizing concept of this manuscript.
\section{Model}
\label{sec:model}
\subsection{Closed-loop transport equation}
\label{subsec:model-equation}
Let
\[
\Om=[l_0,l_m]\Subset(0,\infty),
\qquad
0<l_0<l_m<\infty,
\qquad
T>0 .
\]
The state is
\[
x:[0,T]\times\Om\to[0,\infty),
\qquad
(t,l)\mapsto x(t,l),
\]
and the scalar feedback output is
$\displaystyle E(t)=\int_{l_0}^{l_m}\chi(l)x(t,l)\,dl$.
The closed-loop transport equation is
\begin{equation}
\label{eq:model-state}
\partial_t x(t,l)
+
\partial_l\!\big(g(E(t),l)x(t,l)\big)
=
-\big(\mu(E(t),l)+u(t,l)\big)x(t,l),
\qquad
(t,l)\in(0,T)\times(l_0,l_m).
\end{equation}
The recruitment condition is imposed as an inflow flux:
\[
g(E(t),l_0)x(t,l_0)=p(t),
\qquad
t\in(0,T),
\]
with initial condition
\begin{equation}
\label{eq:model-initial}
x(0,l)=x_0(l),
\qquad
l\in\Om .
\end{equation}
For sufficiently regular \(g\), we have the nondivergence form
\begin{equation}
\label{eq:model-nonconservative}
\partial_t x
+
g(E(t),l)\partial_l x
=
-\big(\partial_l g(E(t),l)+\mu(E(t),l)+u(t,l)\big)x .
\end{equation}
Define the total mass
$\displaystyle M_x(t):=\int_{l_0}^{l_m}x(t,l)\,dl $.
Formally,
\[
\frac{d}{dt}M_x(t)
=
p(t)-g(E(t),l_m)x(t,l_m)
-\int_{l_0}^{l_m}\big(\mu(E(t),l)+u(t,l)\big)x(t,l)\,dl .
\]
Thus the natural feedback box on \([0,T]\) is
\begin{equation}
\label{eq:model-MT}
0\le E(t)\le M_T,
\qquad
M_T:=
\norm{\chi}_{L^\infty(\Om)}
\left(
\norm{x_0}_{L^1(\Om)}+\int_0^T p(s)\,ds
\right).
\end{equation}

As illustrated in Figure~\ref{fig:fig1}, the core mathematical architecture relies on a nonlinear feedback loop. This endogenous feedback distinguishes the problem from a semilinear perturbation and establishes the framework for the Volterra fixed-point approach.

\begin{figure}[htbp]
    \centering
    \includegraphics[width=0.8\textwidth]{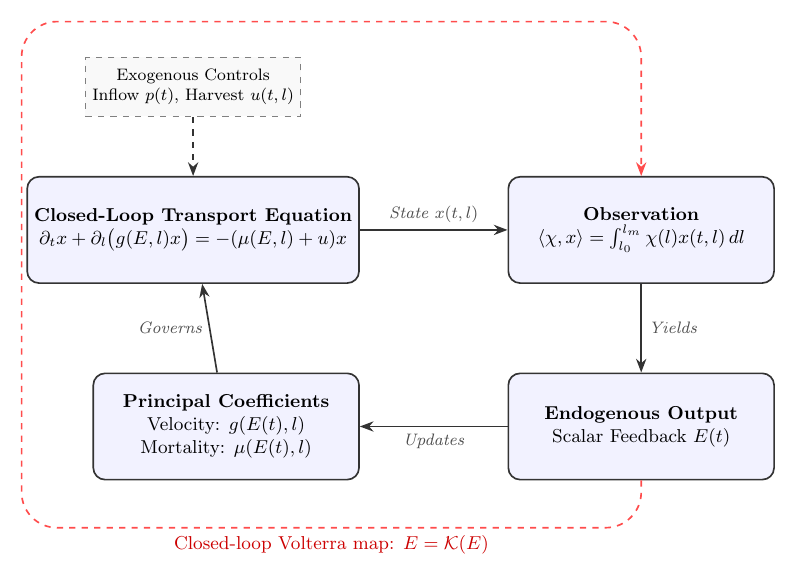}
    \caption{The Closed-Loop Feedback Architecture. A conceptual block diagram illustrating the nonlinear feedback loop. The state density $x(t,l)$ integrates against $\chi(l)$ to produce the scalar output $E(t)$, which then feeds into the principal velocity $g(E,l)$ and mortality $\mu(E,l)$ to dictate the transport operator governing $x(t,l)$.}
    \label{fig:fig1}
\end{figure}

\subsection{Standing assumptions}
\label{subsec:model-assumptions}
All constants below are understood on the parameter box
$\displaystyle [0,M_T]\times\Om $.
\begin{assumption}[Coefficients]
\label{ass:model-coefficients}
For every \(M>0\), the maps
\[
g:[0,\infty)\times\Omega\to(0,\infty),
\qquad
\mu:[0,\infty)\times\Omega\to[0,\infty)
\]
are Borel in \(l\), locally Lipschitz in \(E\), and locally
\(W^{1,\infty}\) in the size variable. On the feedback box
\([0,M_T]\times\Omega\), there exist constants
\(g_M>0\), \(C_M>0\), \(L_M>0\), and \(C_2\ge C_M\) such that, for all
\(E,E'\in[0,M_T]\) and a.e. \(l\in\Omega\),
\begin{align}
&g(E,l)\ge g_M, \label{eq:model-ass-gfloor}\\
&|g(E,l)|+|\mu(E,l)|
+|\partial_Eg(E,l)|+|\partial_lg(E,l)|
+|\partial_E\mu(E,l)|+|\partial_l\mu(E,l)|
\le C_M, \label{eq:model-ass-first}\\
&\|g(E,\cdot)-g(E',\cdot)\|_{W^{1,\infty}(\Omega)}
+
\|\mu(E,\cdot)-\mu(E',\cdot)\|_{W^{1,\infty}(\Omega)}
\le L_M|E-E'|, \label{eq:model-ass-Elip}\\
&\|\partial_{ll}g(E,\cdot)\|_{L^\infty(\Omega)}
+
\|\partial_l\mu(E,\cdot)\|_{L^\infty(\Omega)}
\le C_2 . \label{eq:model-ass-second}
\end{align}
\end{assumption}
\begin{assumption}[Data, observation, and controls]
\label{ass:model-data}
\[
x_0\in BV_+(\Omega)\cap L^\infty_+(\Omega),
\qquad
p\in W^{1,\infty}_+(0,T),
\qquad
\chi\in W^{1,\infty}(\Omega)\cap L^\infty_+(\Omega).
\]
The admissible controls are
\[
\Uad(T):=
\left\{
u\in L^\infty(0,T;W^{1,\infty}(\Om)):
0\le u\le u_{\max}\ \text{a.e.},
\qquad
\norm{\partial_l u}_{L^\infty((0,T)\times\Om)}\le C_2
\right\}.
\]
\end{assumption}
Set
$\displaystyle \Theta_M:=2C_M+u_{\max}$.
Then, on \([0,M_T]\times\Om\),
\begin{equation}
\label{eq:model-basic-bounds}
0<g_M\le g(E,l)\le C_M,
\qquad
0\le \mu(E,l)+u(t,l)\le \Theta_M,
\qquad
\abs{\partial_l g(E,l)}\le C_M .
\end{equation}
For later stationary analysis, when differentiability in \(E\) beyond first order
is needed, we use the additional hypothesis
\begin{equation}
\label{eq:model-C2E}
\norm{\partial_{EE}g(E,\cdot)}_{W^{1,\infty}(\Om)}
+
\norm{\partial_{EE}\mu(E,\cdot)}_{L^\infty(\Om)}
\le C_2',
\qquad
E\in[0,M_T].
\end{equation}
\subsection{Weak solutions and frozen-path operator}
\label{subsec:model-weak}
Let
\[
X:=L^1(\Om),
\qquad
X_+:=L^1_+(\Om),
\qquad
\langle f,h\rangle:=\int_{l_0}^{l_m}f(l)h(l)\,dl .
\]
\begin{definition}[Closed-loop weak solution]
\label{def:model-weak-solution}
A pair
\[
(x,E)\in L^\infty(0,T;L^1_+(\Omega))\times L^\infty_+(0,T)
\]
is a closed-loop weak solution of
\eqref{eq:model-state}--\eqref{eq:model-initial} on \([0,T]\) if
\[
E(t)=\int_{l_0}^{l_m}\chi(l)x(t,l)\,dl
\qquad\text{for a.e. }t\in(0,T),
\]
and, for every
\[
\varphi\in C^1([0,T]\times[l_0,l_m]),
\qquad
\varphi(T,\cdot)=0,
\qquad
\varphi(t,l_m)=0,
\]
one has
\begin{align}
0={}&
\int_0^T\int_{l_0}^{l_m}
x(t,l)
\Big[
\partial_t\varphi(t,l)
+
g(E(t),l)\partial_l\varphi(t,l)
-
(\mu(E(t),l)+u(t,l))\varphi(t,l)
\Big]\,dl\,dt
\nonumber\\
&+
\int_{l_0}^{l_m}x_0(l)\varphi(0,l)\,dl
+
\int_0^T p(t)\varphi(t,l_0)\,dt .
\label{eq:model-weak-form}
\end{align}
\end{definition}
For a prescribed path
\[
E\in L^\infty(0,T),
\qquad
0\le E(t)\le M_T\qquad\text{a.e.},
\]
define the frozen transport problem
\begin{align}
&\partial_t x_E+\partial_l\big(g(E(t),l)x_E\big)
=
-\big(\mu(E(t),l)+u(t,l)\big)x_E,
\label{eq:model-frozen-state}
\\
&g(E(t),l_0)x_E(t,l_0)=p(t),
\qquad
x_E(0,l)=x_0(l).
\label{eq:model-frozen-bc}
\end{align}
The frozen-path solution operator is denoted by
$\displaystyle \Tcal_E(x_0,p,u):=x_E$.
The closed-loop output map is
$\displaystyle \Kcal(E)(t)
:=
\int_{l_0}^{l_m}\chi(l)\Tcal_E(x_0,p,u)(t,l)\,dl$.
The closed-loop problem is
$\displaystyle  E=\Kcal(E)$.
Define
$\displaystyle \Bt:=
\left\{
E\in L^\infty(0,T):
0\le E(t)\le M_T\ \text{a.e.}
\right\}$.
Let
\[
M_x(T):=\norm{x_0}_{L^1(\Om)}+\int_0^T p(s)\,ds,
\qquad
K_\infty:=
e^{C_M T}
\max\left\{
\norm{x_0}_{L^\infty(\Om)},
\frac{\norm{p}_{L^\infty(0,T)}}{g_M}
\right\},
\]
and set
\begin{equation}
\label{eq:model-Lambdastar}
\Lambda^*:=
\norm{\chi}_{L^\infty(\Om)}\norm{p}_{L^\infty(0,T)}
+
\Big(
C_M\norm{\chi'}_{L^\infty(\Om)}
+\Theta_M\norm{\chi}_{L^\infty(\Om)}
\Big)M_x(T)
+
C_M\norm{\chi}_{L^\infty(\Om)}K_\infty .
\end{equation}
The iteration class is
$\displaystyle \Btl:=
\left\{
E\in\Bt:
\Lip(E)\le \Lambda^*
\right\}$.
For \(\beta>0\), define the Bielecki norm
\begin{equation}
\label{eq:model-Bielecki}
\norm{E}_\beta
:=
\operatorname*{ess\,sup}_{0\le t\le T}
e^{-\beta t}\abs{E(t)} .
\end{equation}
\subsection{Stationary profiles and scalar closure}
\label{subsec:model-stationary}
For stationary controls and inflow,
$u(t,l)=u(l)$ and
$p(t)=p>0$,
a stationary frozen profile \(x_E\) at level \(E\ge0\) solves
\[
\frac{d}{dl}\big(g(E,l)x_E(l)\big)
=
-\big(\mu(E,l)+u(l)\big)x_E(l),
\qquad
g(E,l_0)x_E(l_0)=p .
\]
Thus
\begin{equation}
\label{eq:model-stationary-profile}
x_E(l)
=
\frac{p}{g(E,l)}
\exp\!\left[
-\int_{l_0}^{l}
\frac{\mu(E,\xi)+u(\xi)}{g(E,\xi)}\,d\xi
\right].
\end{equation}
The scalar stationary closure map is
$\displaystyle \Phi(E):=
\int_{l_0}^{l_m}\chi(l)x_E(l)\,dl$.
The closed-loop stationary equation is
\[
E=\Phi(E),
\qquad
H(E):=E-\Phi(E)=0 .
\]
For parameter-dependent families,
\[
g=g(E,l;\eta),
\qquad
\mu=\mu(E,l;\eta),
\qquad
u=u(l;\eta),
\qquad
p=p(\eta),
\]
write
$\displaystyle \Phi(E;\eta)
=
\int_{l_0}^{l_m}\chi(l)x_{E,\eta}(l)\,dl$ and
$\displaystyle H(E,\eta):=E-\Phi(E;\eta)$.
The closure-resonance condition is
\[
H_E(E^*,\eta)=0
\qquad\Longleftrightarrow\qquad
\Phi_E(E^*;\eta)=1 .
\]
For later use, set
\[
\alpha(E,l):=-\frac{\partial_E g(E,l)}{g(E,l)},
\qquad
b(E,l):=
\partial_E\left(\frac{\mu(E,l)+u(l)}{g(E,l)}\right).
\]
\subsection{Stationary harvesting problem}
\label{subsec:model-control}
Let
\[
r\ge0,
\qquad
\pi\in L^\infty_+(\Om),
\qquad
u\in\Uad_\mathrm{stat}:=
\left\{
u\in W^{1,\infty}(\Om):
0\le u\le u_{\max},
\qquad
\norm{u'}_{L^\infty(\Om)}\le C_2
\right\}.
\]
The discounted harvesting functional is
\[
J_r(u)
:=
\int_0^\infty e^{-rt}
\int_{l_0}^{l_m}
\pi(l)u(t,l)x(t,l)\,dl\,dt .
\]
For stationary triples \((x^*,E^*,u)\),
\begin{align}
&\frac{d}{dl}\big(g(E^*,l)x^*(l)\big)
=
-\big(\mu(E^*,l)+u(l)\big)x^*(l),
\label{eq:model-stationary-state}
\\
&g(E^*,l_0)x^*(l_0)=p,
\qquad
E^*=\int_{l_0}^{l_m}\chi(l)x^*(l)\,dl .
\label{eq:model-stationary-state-bc}
\end{align}
Write
$\displaystyle g^*(l):=g(E^*,l)$ and
$\displaystyle \mu^*(l):=\mu(E^*,l)$.
The frozen current-value adjoint operator is
\[
\Lcal_r^*\lambda
:=
-g^*(l)\lambda'(l)
+
\big(r+\mu^*(l)+u(l)\big)\lambda(l),
\qquad
\lambda(l_m)=0.
\]
Its inverse is denoted
$\displaystyle R_r:=(\Lcal_r^*)^{-1}$,
that is,
\[
(R_r f)(l)
=
\int_l^{l_m}
\frac{f(s)}{g^*(s)}
\exp\!\left[
-\int_l^s
\frac{r+\mu^*(\xi)+u(\xi)}{g^*(\xi)}\,d\xi
\right]ds.
\]
For the harvesting payoff,
$\displaystyle 
q_{\rm red}(l):=\pi(l)u(l)$.
Define
\[
a(l):=(\partial_E g)(E^*,l)x^*(l),
\qquad
m(l):=(\partial_E\mu)(E^*,l)x^*(l),
\qquad
a_0:=a(l_0).
\]
The corrected feedback co-load is the distribution
$\displaystyle \sigma
:=
-\frac{d}{dl}a-m-a_0\delta_{l_0}$.
Equivalently, for admissible test functions \(\lambda\) with \(\lambda(l_m)=0\),
\[
\langle\sigma,\lambda\rangle
=
\int_{l_0}^{l_m}
x^*(l)
\Big[
(\partial_E g)(E^*,l)\lambda'(l)
-
(\partial_E\mu)(E^*,l)\lambda(l)
\Big]\,dl.
\]
The full stationary adjoint equation has rank-one form
$\displaystyle \big(\Lcal_r^*-\chi\langle\sigma,\cdot\rangle\big)\lambda
=
q_{\rm red}$.
Set
\[
\lambda_{\rm red}:=R_r q_{\rm red},
\qquad
\psi_r:=R_r\chi,
\qquad
A(r):=\langle\sigma,\lambda_{\rm red}\rangle,
\qquad
B(r):=\langle\sigma,\psi_r\rangle .
\]
The stationary switching function is
$\displaystyle S(l):=\pi(l)-\lambda(l)$.
On nonsingular arcs,
$\displaystyle 
u(l)=
\begin{cases}
u_{\max}, & S(l)>0,\\
0, & S(l)<0.
\end{cases}$
\section{Main Results}
\label{sec:main-results}
\subsection{Closed-loop well-posedness}
\label{subsec:main-wellposedness}
The closed-loop equation is reduced to the scalar Volterra fixed point
\[
E=\Kcal(E),
\qquad
\Kcal(E)(t)=\langle \chi,\Tcal_E(x_0,p,u)(t,\cdot)\rangle .
\]
\begin{theorem}[Frozen-path transport estimates]
\label{thm:main-frozen}
Let \(T>0\), \(E\in\Bt\), and suppose
Assumptions~\ref{ass:model-coefficients}--\ref{ass:model-data} hold. Then the
frozen problem \eqref{eq:model-frozen-state}--\eqref{eq:model-frozen-bc} has a
unique nonnegative weak solution
$\displaystyle 
x_E=\Tcal_E(x_0,p,u)\in L^\infty(0,T;L^1_+(\Om))$.
Moreover,
\[
\norm{x_E(t,\cdot)}_{L^1(\Om)}
\le
\norm{x_0}_{L^1(\Om)}+\int_0^t p(s)\,ds,
\]
\[
\norm{x_E(t,\cdot)}_{L^\infty(\Om)}
\le
K_\infty
:=
e^{C_M T}
\max\left\{
\norm{x_0}_{L^\infty(\Om)},
\frac{\norm{p}_{L^\infty(0,T)}}{g_M}
\right\}.
\]
If \(E\in\Btl\), then
$\displaystyle 
x_E(t,\cdot)\in BV(\Om)$ with 
$\displaystyle \sup_{0\le t\le T}\abs{x_E(t,\cdot)}_{BV(\Om)}
\le C_{BV}$.
\end{theorem}
\begin{lemma}[Output regularity]
\label{lem:main-output-regularity}
For every \(E\in\Bt\),
$\Kcal(E)\in W^{1,\infty}(0,T)$ and
$\Lip(\Kcal(E))\le \Lambda^*$,
where \(\Lambda^*\) is given by \eqref{eq:model-Lambdastar}. Hence
$\displaystyle \Kcal(\Bt)\subset\Btl$.
\end{lemma}
\begin{proposition}[Volterra comparison estimate]
\label{prop:main-volterra}
Let \(E_1,E_2\in\Btl\), and set
$\displaystyle x_i:=\Tcal_{E_i}(x_0,p,u)$ for $i=1,2$.
Then
\[
\norm{x_1(t,\cdot)-x_2(t,\cdot)}_{L^1(\Om)}
\le
C_T\int_0^t\abs{E_1(s)-E_2(s)}\,ds ,
\qquad
0\le t\le T,
\]
where one may take
\[
C_T
:=
L_M
\left[
\frac{\norm{p}_{L^\infty(0,T)}}{g_M}
+
2M_x(T)
+
C_{BV}
\right],
\qquad
M_x(T):=\norm{x_0}_{L^1(\Om)}+\int_0^T p(s)\,ds .
\]
Consequently,
\[
\abs{\Kcal(E_1)(t)-\Kcal(E_2)(t)}
\le
\widehat C_T
\int_0^t\abs{E_1(s)-E_2(s)}\,ds,
\qquad
\widehat C_T:=\norm{\chi}_{L^\infty(\Om)}C_T .
\]
\end{proposition}

\begin{theorem}[Closed-loop well-posedness]
\label{thm:main-wellposedness}
Assume Assumptions~\ref{ass:model-coefficients}--\ref{ass:model-data}. Let
$\displaystyle 
\beta>\widehat C_T$.
Then
$\displaystyle 
\Kcal:\Btl\to\Btl$
is a contraction in the Bielecki norm \eqref{eq:model-Bielecki}:
\[
\norm{\Kcal(E_1)-\Kcal(E_2)}_\beta
\le
\frac{\widehat C_T}{\beta}
\norm{E_1-E_2}_\beta.
\]
Hence there exists a unique \(E^{\rm cl}\in\mathcal B_T^{\rm Lip}\) such that
$E^{\rm cl}=\mathcal K(E^{\rm cl})$.
With
$x^{\rm cl}:=\mathcal T_{E^{\rm cl}}(x_0,p,u)$,
the pair \((x^{\rm cl},E^{\rm cl})\) is the unique nonnegative closed-loop weak
solution on \([0,T]\). Moreover,
\[
0\le E^{\rm cl}(t)\le M_T
\qquad\text{for a.e. }t\in(0,T).
\]
\end{theorem}
\begin{proposition}[Continuation criterion]
\label{prop:main-continuation}
Let \((x,E)\) be a closed-loop solution on a maximal interval
$[0,\tau_{\max})$,
$0<\tau_{\max}\le\infty$.
If, for every \(T<\tau_{\max}\),
$\displaystyle 
x\in L^\infty(0,T;BV_+(\Om))$,
and
\begin{equation}
\label{eq:main-continuation-bound}
\sup_{T<\tau_{\max}}
\left[
\norm{x}_{L^\infty(0,T;BV)}
+
\norm{E}_{L^\infty(0,T)}
+
\norm{p}_{W^{1,\infty}(0,T)}
+
\norm{u}_{L^\infty(0,T;BV)}
\right]<\infty,
\end{equation}
with
\begin{equation}
\label{eq:main-continuation-floor}
\inf_{\substack{0\le E\le \sup_{t<\tau_{\max}}E(t)\\ l\in\Om}}
g(E,l)>0,
\end{equation}
then \(\tau_{\max}=\infty\). Hence finite-time breakdown implies failure of at
least one bound in \eqref{eq:main-continuation-bound}--\eqref{eq:main-continuation-floor}.
\end{proposition}
\subsection{Stationary closure and resonance}
\label{subsec:main-stationary}
For stationary \(u=u(l)\) and \(p>0\),
$\displaystyle 
x_E(l)=
\frac{p}{g(E,l)}
\exp\!\left[
-\int_{l_0}^{l}
\frac{\mu(E,\xi)+u(\xi)}{g(E,\xi)}\,d\xi
\right]$, and
$\Phi(E)=\langle\chi,x_E\rangle$.
\begin{lemma}[Derived stationary output interval]
\label{lem:main-stationary-box}
Assume
$g(E,l)\ge g_\flat>0$ and
$(E,l)\in[0,\bar M]\times\Om$,
where
$\displaystyle 
\bar M:=
\norm{\chi}_{L^\infty(\Om)}
(l_m-l_0)\frac{p}{g_\flat}$.
Then
$\displaystyle 
\Phi([0,\bar M])\subset[0,\bar M]$.
\end{lemma}
\begin{proposition}[Closure derivative decomposition]
\label{prop:main-decomp}
Assume \(E\mapsto g(E,\cdot)\) and \(E\mapsto\mu(E,\cdot)\) are \(C^1\) into
\(L^\infty(\Om)\). Let
\[
\alpha(E,l):=-\frac{\partial_E g(E,l)}{g(E,l)},
\qquad
b(E,l):=\partial_E\left(\frac{\mu(E,l)+u(l)}{g(E,l)}\right),
\]
and
$\displaystyle 
W_E(\xi):=\int_\xi^{l_m}\chi(l)x_E(l)\,dl$.
Then
$\displaystyle 
\Phi'(E)=\Rcal(E)-\Ccal(E)$,
where
\[
\Rcal(E):=
\int_{l_0}^{l_m}\chi(l)x_E(l)\alpha(E,l)\,dl,
\qquad
\Ccal(E):=
\int_{l_0}^{l_m}b(E,\xi)W_E(\xi)\,d\xi .
\]
\end{proposition}

\begin{theorem}[Stationary closure and sharp uniqueness]
\label{thm:main-stationary}
Assume Lemma~\ref{lem:main-stationary-box} and the continuity hypotheses of
Proposition~\ref{prop:main-decomp}. Then the stationary closure equation
$E=\Phi(E)$
has at least one solution \(E^*\in[0,\bar M]\). If, in addition,
\(\Phi\in C^1([0,\bar M])\) and
\[
\Phi'(E)<1
\qquad \text{for every }E\in[0,\bar M],
\]
then this solution is unique.
\end{theorem}
\begin{corollary}[Integral dominance]
\label{cor:main-dominance}
If for every $E\in[0,\bar M]$,
$\Ccal(E)>\Rcal(E)$,
then $\Phi'(E)<0<1$ for all $E\in[0,\bar M]$,
and the stationary feedback level is unique.
\end{corollary}

\begin{theorem}[Local fold at closure resonance]
\label{thm:main-fold}
Let
\[
H(E,\eta):=E-\Phi(E;\eta),
\]
and suppose \(H\in C^2\) in a neighborhood of \((E_0,\eta_0)\). Assume
\[
H(E_0,\eta_0)=0,
\qquad
H_E(E_0,\eta_0)=0,
\qquad
H_{EE}(E_0,\eta_0)\ne0,
\qquad
H_\eta(E_0,\eta_0)\ne0 .
\]
Equivalently,
\[
E_0=\Phi(E_0;\eta_0),
\qquad
\Phi_E(E_0;\eta_0)=1,
\qquad
\Phi_{EE}(E_0;\eta_0)\ne0,
\qquad
\Phi_\eta(E_0;\eta_0)\ne0 .
\]
Then the zero set \(H(E,\eta)=0\) is locally a nondegenerate fold. More
precisely, there exists a \(C^2\) function \(\eta=\eta(E)\), defined near
\(E_0\), such that
\[
\eta(E_0)=\eta_0,
\qquad
\eta'(E_0)=0,
\qquad
\eta''(E_0)
=
-\frac{H_{EE}(E_0,\eta_0)}{H_\eta(E_0,\eta_0)} .
\]
Equivalently,
\[
\eta(E_0+s)
=
\eta_0
-
\frac{\Phi_{EE}(E_0;\eta_0)}
     {2\Phi_\eta(E_0;\eta_0)}
s^2
+
o(s^2).
\]
\end{theorem}

As illustrated in Figure~\ref{fig:fig3}, the geometric solution to the scalar closure equation $E = \Phi(E; \eta)$ provides an intuitive understanding of the closure-resonance threshold. This clearly demonstrates the breakdown of local closure invertibility and the onset of the fold bifurcation.

\begin{figure}[htbp]
    \centering
    \includegraphics[width=0.8\textwidth]{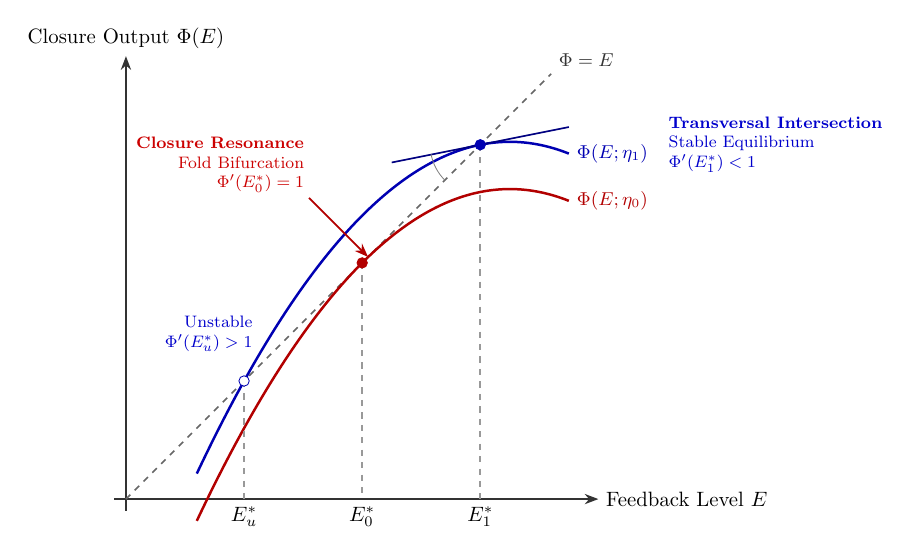}
    \caption{Stationary Closure and the Fold Bifurcation. A graphical solution to the scalar closure equation $E = \Phi(E; \eta)$. The plot features the identity line $\Phi = E$ alongside two curves: one intersecting transversally to indicate a stable equilibrium ($\Phi'(E_1^*) < 1$), and another tangent to the identity line at the closure resonance threshold ($\Phi'(E_0^*) = 1$), marking the fold bifurcation.}
    \label{fig:fig3}
\end{figure}

\subsection{Linearized renewal equation and stability}
\label{subsec:main-stability}
Let \((x^*,E^*)\) be a stationary closed-loop equilibrium. Set
$g^*=g(E^*,\cdot)$,
$\mu^*=\mu(E^*,\cdot)$,
and
$e(t)=\langle\chi,v(t,\cdot)\rangle$.
The linearized equation is
\begin{equation}
\label{eq:main-linearized}
\partial_t v+\partial_l(g^*v)+(\mu^*+u)v
=
\sigma_{\rm int}\,e(t),
\qquad
e(t)=\langle\chi,v(t,\cdot)\rangle ,
\end{equation}
with
\[
g^*(l_0)v(t,l_0)=-a_0e(t),
\qquad
a_0:=(\partial_E g)(E^*,l_0)x^*(l_0),
\]
and
\begin{equation}
\label{eq:main-sigma-int}
\sigma_{\rm int}
:=
-\frac{d}{dl}\big((\partial_E g)(E^*,\cdot)x^*\big)
-(\partial_E\mu)(E^*,\cdot)x^* .
\end{equation}
\begin{lemma}[Finite sweep-out]
\label{lem:main-sweep}
Let \(U_0(t)\) be the zero-inflow transport semigroup generated by
\[
\Lcal_{\rm tr}\varphi
=
-\partial_l(g^*\varphi)-(\mu^*+u)\varphi,
\qquad
D(\Lcal_{\rm tr})
=
\{\varphi\in W^{1,1}(\Om):g^*(l_0)\varphi(l_0)=0\}.
\]
Let
$\displaystyle 
t_\sharp:=\int_{l_0}^{l_m}\frac{d\xi}{g^*(\xi)}$.
Then for all $t\ge t_\sharp$, $U_0(t)=0$.
Moreover, for every \(\lambda\in\C\), the homogeneous problem
\[
\lambda\varphi
=
-\partial_l(g^*\varphi)-(\mu^*+u)\varphi,
\qquad
g^*(l_0)\varphi(l_0)=0,
\]
has only the trivial solution.
\end{lemma}

As depicted in Figure~\ref{fig:fig2}, the space-time diagram illustrates the characteristic curves and the finite sweep-out time $t_\sharp$, visually proving the intrinsic bound and the finite-memory reduction required for the renewal equation.

\begin{figure}[htbp]
    \centering
    \includegraphics[width=0.8\textwidth]{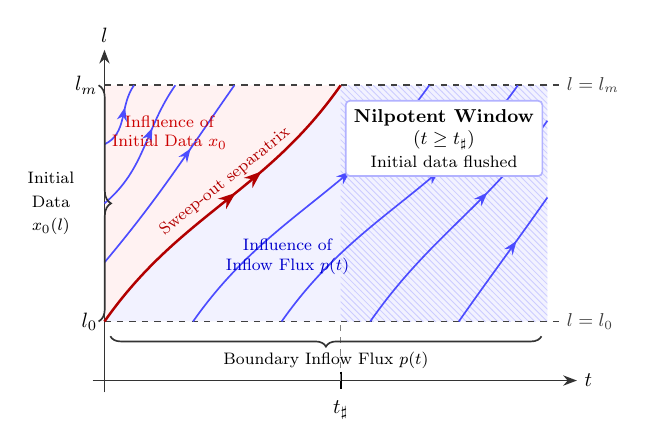}
    \caption{Characteristic Sweep-Out and Intrinsic Bounds. A space-time $(t,l)$ diagram showing the strictly increasing characteristic curves $L(s;t,l)$ flowing from the initial data slice $t=0$ and the inflow boundary $l=l_0$. The shaded nilpotent window ($t \ge t_\sharp$) highlights the region where the initial data $x_0$ has been entirely flushed from the system.}
    \label{fig:fig2}
\end{figure}

\begin{proposition}[Scalar renewal reduction]
\label{prop:main-renewal}
The scalar output perturbation satisfies
\[
e(t)=f(t)+\int_0^t\kappa(t-s)e(s)\,ds,
\qquad
f(t):=\langle\chi,U_0(t)v_0\rangle ,
\]
where
$\displaystyle 
\kappa(\tau)
=
\langle\chi,U_0(\tau)\sigma_{\rm int}\rangle
-
a_0\langle\chi,\Sigma_b(\tau)\rangle$.
Furthermore,
$\operatorname{supp}f\subset[0,t_\sharp]$,
$\operatorname{supp}\kappa\subset[0,t_\sharp]$,
and
$\kappa\in BV([0,t_\sharp])$.
\end{proposition}
\begin{definition}[Characteristic function]
\label{def:main-characteristic}
For \(\lambda\in\C\), let \(\varphi_1^\lambda\) solve
\[
\frac{d}{dl}\big(g^*\varphi_1^\lambda\big)
+
(\lambda+\mu^*+u)\varphi_1^\lambda
=
\sigma_{\rm int},
\qquad
g^*(l_0)\varphi_1^\lambda(l_0)=-a_0 .
\]
Define
$\displaystyle 
\Ev(\lambda)
:=
\langle\chi,\varphi_1^\lambda\rangle$.
\end{definition}
\begin{lemma}[Properties of \(\Ev\)]
\label{lem:main-Ev-props}
The map \(\Ev:\C\to\C\) is entire,
\[
\Ev(\bar\lambda)=\overline{\Ev(\lambda)},
\qquad
\Ev(\lambda)=\int_0^{t_\sharp}e^{-\lambda\tau}\kappa(\tau)\,d\tau .
\]
Moreover,
$\Ev(\lambda)\to0$ as $\lambda\to+\infty$, $\lambda\in\R$,
and, for \(\Re\lambda\ge0\),
$\displaystyle 
\abs{\Ev(\lambda)}\le G_0$,
where
\[
G_0:=
\int_{l_0}^{l_m}
\frac{\chi(l)}{g^*(l)}
\Pi^0(l)
\left[
\abs{a_0}
+
\int_{l_0}^{l}
\Pi^0(s)^{-1}\,d\abs{\sigma_{\rm int}}(s)
\right]dl,
\]
with
$\displaystyle 
\Pi^0(l):=
\exp\!\left[
-\int_{l_0}^{l}\frac{\mu^*(\xi)+u(\xi)}{g^*(\xi)}\,d\xi
\right]$.
Finally,
$\displaystyle 
\Ev(0)=\Phi'(E^*)$.
\end{lemma}
\begin{theorem}[Linearized spectral characterization and growth bound]
\label{thm:main-spectrum}
Let
$\omega_0
:=
\sup\{\Re\lambda:\Ev(\lambda)=1\}$ with the convention
$\sup\varnothing:=-\infty$.
Then the nontrivial feedback point spectrum of the linearized generator is
$\displaystyle 
\{\lambda\in\C:\Ev(\lambda)=1\}$.
The root set is discrete and finite in every half-plane
$\displaystyle 
\{\Re\lambda\ge-\omega\}$.
The scalar renewal feedback dynamics has growth bound \(\omega_0\). Hence the following conclusions hold:
\[
\omega_0<0
\qquad\Longrightarrow\qquad
\text{linearized exponential stability},\qquad \text{and} \qquad
\omega_0>0
\qquad\Longrightarrow\qquad
\text{linearized instability}.
\]
\end{theorem}
\begin{theorem}[Instability above the closure threshold]
\label{thm:main-instability}
If
$\displaystyle 
\Phi'(E^*)>1$,
then there exists \(\lambda^*>0\) such that
$\displaystyle 
\Ev(\lambda^*)=1$.
Consequently,
$\displaystyle 
\omega_0\ge\lambda^*>0$,
and the equilibrium is linearly unstable.
\end{theorem}
\begin{theorem}[Linearized stability below the total gain threshold]
\label{thm:main-gain-stability}
If
$\displaystyle 
G_0<1$,
then for any $\lambda$ such that $\Re\lambda\ge0$,
$\Ev(\lambda)\ne1$.
Hence
$\displaystyle 
\omega_0<0$,
and the linearized feedback dynamics is exponentially stable. Moreover,
\[
G_0<1
\qquad\Longrightarrow\qquad
\abs{\Phi'(E^*)}<1.
\]
If the quadratic nonlinear remainder estimate of
Theorem~\ref{thm:main-nonlinear-stability} holds, then the nonlinear closed-loop
equilibrium is locally exponentially asymptotically stable.
\end{theorem}

As illustrated in Figure~\ref{fig:fig4}, the map of the point spectrum in the complex plane highlights the spectral growth bound $\omega_0$. Crucially, the root at the origin connects to our headline identity $\mathcal{E}(0) = \Phi'(E^*)$, visually demonstrating that the continuous system destabilizes as the real root crosses into the right half-plane.

\begin{figure}[htbp]
    \centering
    \includegraphics[width=0.8\textwidth]{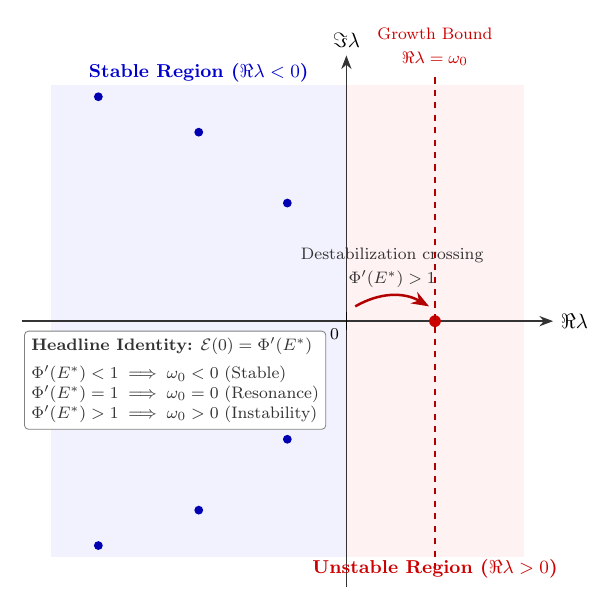}
    \caption{The Feedback Spectrum and Zero-Crossing. A map of the point spectrum in the complex $\lambda$-plane illustrating the spectral growth bound $\Re\lambda = \omega_0$ (dashed line). The relationship $\mathcal{E}(0) = \Phi'(E^*)$ directly connects the scalar derivative from the stationary closure to the dynamical instability of the continuous system.}
    \label{fig:fig4}
\end{figure}

\begin{theorem}[Nonlinear stability under quadratic remainder]
\label{thm:main-nonlinear-stability}
Assume
$\displaystyle 
\omega_0<0$.
Let \(\omega\in(\omega_0,0)\). Suppose the nonlinear scalar remainder
\(\Gcal[e,v_0]\) satisfies
\begin{equation}
\label{eq:main-quadratic-remainder}
\norm{\Gcal[e,v_0]}_\omega
\le
C\big(\norm{v_0}_{BV}+\norm{e}_\omega\big)^2
\end{equation}
for all sufficiently small \(\norm{v_0}_{BV}\) and \(\norm{e}_\omega\). Then
there exist \(\delta>0\) and \(M\ge1\) such that
\[
\norm{v_0}_{BV}\le\delta
\qquad\Longrightarrow\qquad
\norm{x(t)-x^*}_{L^1(\Om)}+\abs{E(t)-E^*}
\le
M e^{\omega t}\norm{v_0}_{BV},
\qquad
t\ge0 .
\]
If \eqref{eq:model-C2E} holds, then \eqref{eq:main-quadratic-remainder} is valid.
\end{theorem}
\subsection{Rank-one adjoint reduction}
\label{subsec:main-adjoint}
Let
$R_r:=(\Lcal_r^*)^{-1}$,
$\lambda_{\rm red}:=R_rq_{\rm red}$,
$\psi_r:=R_r\chi$,
$A(r):=\langle\sigma,\lambda_{\rm red}\rangle$,
and
$B(r):=\langle\sigma,\psi_r\rangle$.
\begin{lemma}[Admissibility of the loop functional]
\label{lem:main-loop-admissibility}
Let
$\displaystyle 
X_*:=
\{\lambda\in W^{1,1}(\Om):\lambda(l_m)=0\}$.
For every \(f\in L^1(\Om)\),
$R_rf\in X_*$
and
$\displaystyle 
\norm{R_rf}_{L^\infty(\Om)}
\le
g_M^{-1}\norm{f}_{L^1(\Om)}$.
Moreover,
\[
\norm{(R_rf)'}_{L^1(\Om)}
\le
g_M^{-1}
\left[
(r+C_M+u_{\max})(l_m-l_0)g_M^{-1}+1
\right]
\norm{f}_{L^1(\Om)}.
\]
The functional $\displaystyle 
f\mapsto \langle\sigma,R_rf\rangle$ is bounded on \(L^1(\Om)\).
\end{lemma}

The feedback dependence of \(g\) and \(\mu\) enters the stationary adjoint equation
only through the scalar variation of the output \(E=\langle\chi,x\rangle\).
Consequently, the full adjoint equation is a rank-one perturbation of the frozen
adjoint equation. This yields an explicit Sherman--Morrison-type formula. The
same denominator \(1-B(r)\) that appears in the adjoint correction is the
discounted version of the zero-frequency feedback margin.

\begin{theorem}[Rank-one adjoint formula]
\label{thm:main-rankone}
Let
\[
\mathcal A_r:=\mathcal L_r^*-\chi\langle\sigma,\cdot\rangle,
\qquad
R_r:=(\mathcal L_r^*)^{-1},
\qquad
\psi_r:=R_r\chi .
\]
For \(q\in L^1(\Omega)\), set
\[
\lambda_q:=R_rq,
\qquad
B(r):=\langle\sigma,\psi_r\rangle .
\]
If $1-B(r)\ne0$,
then the equation
\[
\mathcal A_r\lambda=q,
\qquad
\lambda(l_m)=0,
\]
has the unique solution
\[
\lambda
=
\lambda_q
+
\frac{\langle\sigma,\lambda_q\rangle}
     {1-B(r)}
\psi_r .
\]
If \(1-B(r)=0\), then
$\psi_r\in\ker \mathcal A_r$.
Moreover, solvability in the singular case requires the compatibility condition
\[
\langle\sigma,R_rq\rangle=0,
\]
and solutions, when they exist, are nonunique modulo
\(\operatorname{span}\{\psi_r\}\).
\end{theorem}

\begin{corollary}[Stationary harvesting adjoint and switching function]
\label{cor:main-switching-formula}
For \(q=q_{\rm red}=\pi u\), if \(1-B(r)\ne0\), then
\[
\lambda
=
\lambda_{\rm red}
+
\Gamma_r\psi_r,
\qquad
\Gamma_r:=
\frac{A(r)}{1-B(r)}.
\]
The switching function is
\[
S(l)
=
S_{\rm red}(l)-\Gamma_r\psi_r(l),
\qquad
S_{\rm red}(l):=\pi(l)-\lambda_{\rm red}(l).
\]
Thus
\[
u(l)=u_{\max}\qquad\text{on }\{S>0\},
\qquad
u(l)=0\qquad\text{on }\{S<0\}.
\]
\end{corollary}
\subsection{Forward--adjoint identity and switching geometry}
\label{subsec:main-identity-switching}
\begin{theorem}[Forward--adjoint feedback identity]
\label{thm:main-identity}
Let \(r=0\). Let
$\psi_0:=R_0\chi$ and
$B(0):=\langle\sigma,\psi_0\rangle$.
Then
$\displaystyle 
B(0)=\Phi'(E^*)$.
Consequently,
\[
\Ev(0)=\Phi'(E^*)=B(0),
\qquad
1-\Ev(0)=1-\Phi'(E^*)=1-B(0).
\]
\end{theorem}
\begin{corollary}[Discount perturbation]
\label{cor:main-discount}
Let
$\displaystyle 
B(r):=\langle\sigma,R_r\chi\rangle$.
Then, on compact stationary sets, as $r\downarrow0$,
$\displaystyle 
B(r)=B(0)+\mathcal O(r) = \Phi'(E^*)+\mathcal O(r)$.
\end{corollary}
\begin{proposition}[No exterior switching points]
\label{prop:main-no-exterior}
Assume
$S_{\rm red}\in C(\Om)$,
$\psi_r\in C(\Om)$,
and \(S_{\rm red}\) has a unique zero \(l_{\rm red}\in(l_0,l_m)\). For
$\displaystyle 
\mathcal N_\delta:=(l_{\rm red}-\delta,l_{\rm red}+\delta)\cap\Om$,
define
$\displaystyle 
m_{\rm out}(\delta):=
\inf_{\Om\setminus\mathcal N_\delta}\abs{S_{\rm red}(l)}$.
If
$\displaystyle 
\abs{\Gamma_r}\norm{\psi_r}_{L^\infty(\Om)}
<
m_{\rm out}(\delta)$,
then for all $l\in\Om\setminus\mathcal N_\delta$,
$S(l)\ne0$.
\end{proposition}
\begin{proposition}[Persistence of a single switching threshold]
\label{prop:main-single-threshold}
Assume
$\displaystyle 
S_{\rm red},\psi_r\in W^{1,\infty}(\Om)$,
\(S_{\rm red}\) has a unique simple zero \(l_{\rm red}\in(l_0,l_m)\), and
$\displaystyle 
S_{\rm red}'(l_{\rm red})\ne0$.
Let
$\displaystyle 
m_1(\delta):=
\inf_{\mathcal N_\delta}\abs{S_{\rm red}'(l)}$.
If
\[
\abs{\Gamma_r}\norm{\psi_r}_{L^\infty(\Om)}
<
m_{\rm out}(\delta),
\qquad
\abs{\Gamma_r}\norm{\psi_r'}_{L^\infty(\Om)}
<
m_1(\delta),
\]
then \(S\) has at most one zero in \(\Om\). If, in addition,
$\displaystyle 
S(l_0)S(l_m)<0$,
then \(S\) has exactly one zero \(l_*\in\mathcal N_\delta\), and
$\displaystyle 
u(l)=u_{\max}\mathbf 1_{\{S(l)>0\}}$
is a single-threshold policy.
\end{proposition}
\begin{corollary}[First-order switching displacement]
\label{cor:main-switch-shift}
If \(l_*\) exists and \(S'(l_*)\ne0\), then
$\displaystyle 
l_*=
l_{\rm red}
+
\Delta l$,
with
$\displaystyle 
\Delta l
=
\frac{\Gamma_r\psi_r(l_{\rm red})}
     {S_{\rm red}'(l_{\rm red})} + \mathcal O(\Gamma_r^2)$.
\end{corollary}
\begin{proposition}[Window creation by tangency]
\label{prop:main-window}
Additional switching zeros can be created or annihilated only at solutions of
\[
S_{\rm red}(l)=\Gamma_r\psi_r(l),
\qquad
S_{\rm red}'(l)=\Gamma_r\psi_r'(l).
\]
Equivalently,
$\displaystyle 
S_{\rm red}(l)\psi_r'(l)-S_{\rm red}'(l)\psi_r(l)=0$.
\end{proposition}

As illustrated in Figure~\ref{fig:fig5}, the rank-one adjoint correction translates the abstract Sherman-Morrison-type formula into a tangible optimal control consequence. Specifically, the figure demonstrates the geometric displacement of bang-bang switching thresholds, highlighting the structural shift of the harvesting zone from the unperturbed threshold $l_{\rm red}$ to the new threshold $l_*$.

\begin{figure}[htbp]
    \centering
    \includegraphics[width=0.8\textwidth]{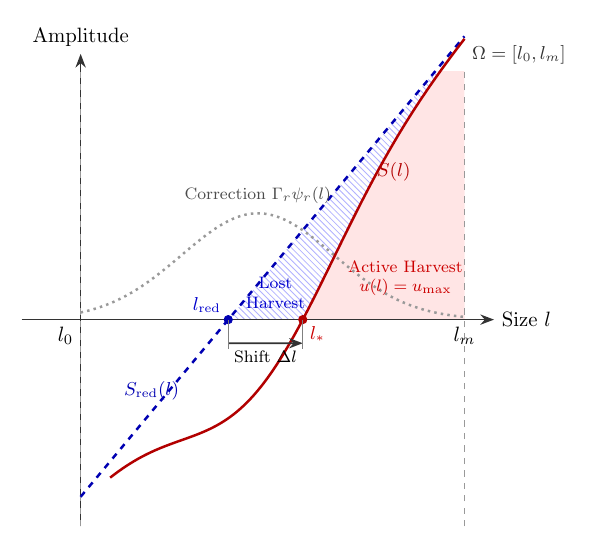}
    \caption{Rank-One Adjoint Correction and Switching Geometry. A plot over the size domain $\Omega$ showing the unperturbed switching function $S_{\rm red}(l)$ crossing zero at $l_{\rm red}$, and the rank-one correction envelope $\Gamma_r \psi_r(l)$. The perturbed switching function $S(l)$ illustrates the shift of the bang-bang threshold from $l_{\rm red}$ to the new $l_*$, highlighting the structurally shifted and lost harvest zones.}
    \label{fig:fig5}
\end{figure}

\subsection{A posteriori diagnostics}
\label{subsec:main-diagnostics}
The scalar quantities
$H(E)=E-\Phi(E)$,
$B(0)-\Phi'(E^*)$,
$1-B(r)$,
and
$S_{\rm red}-\Gamma_r\psi_r$
generate computable diagnostics.
\begin{proposition}[Closure residual bound]
\label{prop:main-closure-residual}
Let
$\Phi\in C^1(I)$,
$I\subset[0,\bar M]$,
and
$\tilde E\in I$,
and set
$\displaystyle 
\rho_0:=\abs{\tilde E-\Phi(\tilde E)}$.
Assume for all $E\in I$,
$H'(E)=1-\Phi'(E)\ge m>0$,
and
$\displaystyle 
\left[
\tilde E-\frac{\rho_0}{m},
\tilde E+\frac{\rho_0}{m}
\right]
\subset I$.
Then \(H(E)=0\) has a unique zero \(E^*\in I\), and
$\displaystyle 
\abs{E^*-\tilde E}
\le
\frac{\rho_0}{m}
=
\frac{\abs{\tilde E-\Phi(\tilde E)}}{m}$.
\end{proposition}
\begin{definition}[Forward--adjoint consistency defect]
\label{def:main-FA-defect}
For a discretization producing
$B^h(0)$ and
$\Phi_h'(E^*)$,
define
$\mathcal D_{\rm FA}^h
:=
B^h(0)-\Phi_h'(E^*)$.
The exact defect is
$\displaystyle 
\mathcal D_{\rm FA}
:=
B(0)-\Phi'(E^*)=0$.
\end{definition}
\begin{proposition}[Switching robustness certificate]
\label{prop:main-switching-certificate}
Let computed quantities
$\widehat B$, 
$\widehat A$, 
$\widehat\psi$, 
$\widehat\psi'$,
$\widehat\Delta:=1-\widehat B$,
and
$\displaystyle \widehat\Gamma:=\frac{\widehat A}{\widehat\Delta}$
satisfy
$\displaystyle 
\widehat\Delta\ne0$.
Let
$\widehat\rho_\psi:=
\abs{\widehat\Gamma}\norm{\widehat\psi}_{L^\infty(\Om)}$ and
$\widehat\rho_{\psi,1}:=
\abs{\widehat\Gamma}\norm{\widehat\psi'}_{L^\infty(\Om)}$.
If
\begin{equation}
\label{eq:main-switch-certificate}
\widehat\rho_\psi<\widehat m_{\rm out}(\delta),
\qquad
\widehat\rho_{\psi,1}<\widehat m_1(\delta),
\end{equation}
then, up to the verified accuracy of the computed quantities, the switching
policy is single-threshold. Failure of \eqref{eq:main-switch-certificate},
together with a solution of
$\displaystyle 
\widehat S_{\rm red}\widehat\psi'
-
\widehat S_{\rm red}'\widehat\psi
=0$,
indicates creation or annihilation of a switching window.
\end{proposition}
\section{Proofs}
\label{sec:proofs}
The frozen-path estimates used below are collected in
Appendix~\ref{app:transport}. In particular,
Theorem~\ref{thm:frozen-transport} gives positivity, the mass estimate,
the output Lipschitz bound, the \(BV\)-propagation estimate, and the
Volterra comparison estimate.
\subsection{Proof of closed-loop well-posedness}
\label{subsec:proof-wp}
\begin{proof}[Proof of Theorem~\ref{thm:main-frozen}]
This is Theorem~\ref{thm:frozen-transport} in Appendix~\ref{app:transport}.
\end{proof}
\begin{proof}[Proof of Lemma~\ref{lem:main-output-regularity}]
By Theorem~\ref{thm:frozen-transport}, for each \(E\in\Bt\),
$x_E=\Tcal_E(x_0,p,u)\ge0$, 
$\norm{x_E(t)}_{L^1}\le M_x(T)$, and
$\norm{x_E(t)}_{L^\infty}\le K_\infty$.
Testing the frozen equation against \(\chi\) gives, for a.e. \(t\),
\[
\frac{d}{dt}\Kcal(E)(t)
=
\chi(l_0)p(t)
-\chi(l_m)g(E(t),l_m)x_E(t,l_m)
+\int_{\Om}\big(g(E(t),l)\chi'(l)-(\mu(E(t),l)+u(t,l))\chi(l)\big)x_E(t,l)\,dl .
\]
Hence
\[
\abs{\dot{\Kcal}(E)(t)}
\le
\norm{\chi}_{\infty}\norm{p}_{\infty}
+
C_M\norm{\chi}_{\infty}K_\infty
+
\big(C_M\norm{\chi'}_{\infty}+\Theta_M\norm{\chi}_{\infty}\big)M_x(T)
=\Lambda^* .
\]
Thus \(\Kcal(E)\in W^{1,\infty}(0,T)\), \(\Lip(\Kcal(E))\le\Lambda^*\), and
$\displaystyle 
0\le \Kcal(E)(t)\le \norm{\chi}_{\infty}M_x(T)=M_T$.
Therefore \(\Kcal(\Bt)\subset\Btl\).
\end{proof}
\begin{proof}[Proof of Proposition~\ref{prop:main-volterra}]
This is Proposition~\ref{prop:volterra-transport} in Appendix~\ref{app:transport}.
The estimate has the form
\[
\norm{\Tcal_{E_1}(x_0,p,u)(t)-\Tcal_{E_2}(x_0,p,u)(t)}_{L^1}
\le
C_T\int_0^t\abs{E_1(s)-E_2(s)}\,ds .
\]
Multiplying by \(\norm{\chi}_{\infty}\) gives
\[
\abs{\Kcal(E_1)(t)-\Kcal(E_2)(t)}
\le
\norm{\chi}_{\infty}C_T\int_0^t\abs{E_1(s)-E_2(s)}\,ds
=
\widehat C_T\int_0^t\abs{E_1(s)-E_2(s)}\,ds .
\]
\end{proof}
\begin{proof}[Proof of Theorem~\ref{thm:main-wellposedness}]
By Lemma~\ref{lem:main-output-regularity},
$\displaystyle 
\Kcal:\Btl\to\Btl$.
The set \(\Btl\) is closed in \(L^\infty(0,T)\): if \(E_n\to E\) uniformly,
\(0\le E_n\le M_T\), and \(\Lip(E_n)\le\Lambda^*\), then
\(0\le E\le M_T\) and
\[
\abs{E(t)-E(s)}
=
\lim_{n\to\infty}\abs{E_n(t)-E_n(s)}
\le
\Lambda^*\abs{t-s}.
\]
Thus \((\Btl,\norm{\cdot}_\beta)\) is complete.
For \(E_1,E_2\in\Btl\), Proposition~\ref{prop:main-volterra} gives
\[
\abs{\Kcal(E_1)(t)-\Kcal(E_2)(t)}
\le
\widehat C_T\int_0^t\abs{E_1(s)-E_2(s)}\,ds .
\]
Since \(\abs{E_1(s)-E_2(s)}\le e^{\beta s}\norm{E_1-E_2}_\beta\),
\[
e^{-\beta t}\abs{\Kcal(E_1)(t)-\Kcal(E_2)(t)}
\le
\widehat C_T e^{-\beta t}\int_0^t e^{\beta s}\,ds\,
\norm{E_1-E_2}_\beta
\le
\frac{\widehat C_T}{\beta}\norm{E_1-E_2}_\beta .
\]
Therefore
$\displaystyle 
\norm{\Kcal(E_1)-\Kcal(E_2)}_\beta
\le
\frac{\widehat C_T}{\beta}\norm{E_1-E_2}_\beta$.
If \(\beta>\widehat C_T\), then \(\widehat C_T/\beta<1\). Banach's theorem gives
a unique \(E^{\mathrm{cl}}\in\Btl\) with \(E^{\mathrm{cl}}=\Kcal(E^{\mathrm{cl}})\). Put
\(x^{\mathrm{cl}}:=\Tcal_{E^{\mathrm{cl}}}(x_0,p,u)\). Then
\(E^{\mathrm{cl}}=\langle\chi,x^{\mathrm{cl}}\rangle\), and the frozen weak formulation is
exactly \eqref{eq:model-weak-form}; hence \((x^{\mathrm{cl}},E^{\mathrm{cl}})\) is a
closed-loop weak solution.
If \((x,E)\in L^\infty(0,T;L^1_+)\times L^\infty_+\) is another closed-loop weak
solution, then \(x=\Tcal_E(x_0,p,u)\) and \(E=\Kcal(E)\). Lemma~\ref{lem:main-output-regularity}
implies \(E\in\Btl\). The fixed point is unique in \(\Btl\), so \(E=E^{\mathrm{cl}}\), hence
\(x=\Tcal_{E^{\mathrm{cl}}}(x_0,p,u)=x^{\mathrm{cl}}\).
\end{proof}
\begin{proof}[Proof of Proposition~\ref{prop:main-continuation}]
Assume \(\tau_{\max}<\infty\) and the bounds
\[
\norm{x}_{L^\infty(0,T;BV)}
+
\norm{E}_{L^\infty(0,T)}
+
\norm{p}_{W^{1,\infty}(0,T)}
+
\norm{u}_{L^\infty(0,T;BV)}
\le C_*
\]
hold uniformly for \(T<\tau_{\max}\), with
$\displaystyle 
\inf_{\substack{0\le E\le \sup_{t<\tau_{\max}}E(t)\\ l\in\Om}}g(E,l)\ge g_*>0$.
Choose \(T_0<\tau_{\max}\). Then
$x(T_0,\cdot)\in BV_+(\Om)$ and
$\norm{x(T_0,\cdot)}_{BV}\le C_*$.
Using \(x(T_0,\cdot)\) as initial datum, the constants in
Theorem~\ref{thm:main-wellposedness} on \([T_0,T_0+\delta]\) depend only on
$C_*$, $g_*$, $C_M$, $L_M$, $C_2$, $u_{\max}$, and $\Om$,
not on \(T_0\). Thus \(\delta>0\) can be chosen uniformly for \(T_0\) close to
\(\tau_{\max}\). Taking \(T_0\in(\tau_{\max}-\delta/2,\tau_{\max})\) extends the
solution beyond \(\tau_{\max}\), contradiction. Therefore \(\tau_{\max}=\infty\).
\end{proof}
\subsection{Proofs for stationary closure}
\label{subsec:proof-stationary}
\begin{proof}[Proof of Lemma~\ref{lem:main-stationary-box}]
Fix \(E\in[0,\bar M]\). The stationary frozen equation is
\[
\frac{d}{dl}\big(g(E,l)x_E(l)\big)
=
-\big(\mu(E,l)+u(l)\big)x_E(l),
\qquad
g(E,l_0)x_E(l_0)=p .
\]
Set
$Y_E(l):=g(E,l)x_E(l)$ and
$\displaystyle q_E(l):=\frac{\mu(E,l)+u(l)}{g(E,l)}$.
Then
$Y_E'(l)=-q_E(l)Y_E(l)$ with
$Y_E(l_0)=p$.
Hence
\[
Y_E(l)
=
p\exp\!\left[-\int_{l_0}^{l}q_E(\xi)\,d\xi\right]
=
p\exp\!\left[-\int_{l_0}^{l}
\frac{\mu(E,\xi)+u(\xi)}{g(E,\xi)}\,d\xi\right].
\]
Since
$\mu(E,\xi)\ge0$, $u(\xi)\ge0$, $g(E,\xi)>0$,
one has $0<Y_E(l)\le p$.
Using \(g(E,l)\ge g_\flat\),
$\displaystyle 
0<x_E(l)=\frac{Y_E(l)}{g(E,l)}
\le
\frac{p}{g_\flat}$.
Consequently,
\[
0\le\Phi(E)
=
\int_{l_0}^{l_m}\chi(l)x_E(l)\,dl
\le
\int_{l_0}^{l_m}\norm{\chi}_{\infty}\frac{p}{g_\flat}\,dl
=
\norm{\chi}_{\infty}(l_m-l_0)\frac{p}{g_\flat}
=
\bar M .
\]
Thus
$E\in[0,\bar M]$ implies $\Phi(E)\in[0,\bar M]$, i.e.
$\Phi([0,\bar M])\subset[0,\bar M]$.
\end{proof}
\begin{proof}[Proof of Proposition~\ref{prop:main-decomp}]
For \(E\in[0,\bar M]\), write
$\displaystyle 
q_E(l):=\frac{\mu(E,l)+u(l)}{g(E,l)}$.
Then
$\displaystyle x_E(l)
=
\frac{p}{g(E,l)}
\exp\!\left[-\int_{l_0}^{l}q_E(\xi)\,d\xi\right]$.
Thus
$\displaystyle 
\log x_E(l)
=
\log p-\log g(E,l)-\int_{l_0}^{l}q_E(\xi)\,d\xi$.
Differentiate in \(E\):
\[
\partial_E\log x_E(l)
=
-\frac{\partial_E g(E,l)}{g(E,l)}
-
\int_{l_0}^{l}\partial_E q_E(\xi)\,d\xi .
\]
With
\[
\alpha(E,l):=-\frac{\partial_E g(E,l)}{g(E,l)},
\qquad
b(E,l):=\partial_E q_E(l)
=
\partial_E\left(\frac{\mu(E,l)+u(l)}{g(E,l)}\right),
\]
we get
$\displaystyle 
\partial_E\log x_E(l)
=
\alpha(E,l)-\int_{l_0}^{l}b(E,\xi)\,d\xi$,
and hence
\begin{equation}
\label{eq:proof-stationary-sensitivity}
\partial_E x_E(l)
=
x_E(l)\alpha(E,l)
-
x_E(l)\int_{l_0}^{l}b(E,\xi)\,d\xi .
\end{equation}
Since
$\displaystyle 
\Phi(E)=\int_{l_0}^{l_m}\chi(l)x_E(l)\,dl$,
differentiation under the integral gives
$\displaystyle 
\Phi'(E)
=
\int_{l_0}^{l_m}\chi(l)\partial_E x_E(l)\,dl$.
Using \eqref{eq:proof-stationary-sensitivity},
\[
\begin{aligned}
\Phi'(E)
&=
\int_{l_0}^{l_m}\chi(l)x_E(l)\alpha(E,l)\,dl
-
\int_{l_0}^{l_m}\chi(l)x_E(l)
\left(\int_{l_0}^{l}b(E,\xi)\,d\xi\right)dl\\[4pt]
&=\int_{l_0}^{l_m}\chi(l)x_E(l)\alpha(E,l)\,dl
- \int_{l_0}^{l_m}
b(E,\xi)
\left(
\int_{\xi}^{l_m}\chi(l)x_E(l)\,dl
\right)d\xi\\[4pt]
&=
\underbrace{
\int_{l_0}^{l_m}\chi(l)x_E(l)\alpha(E,l)\,dl
}_{\Rcal(E)}
-
\underbrace{
\int_{l_0}^{l_m}b(E,\xi)W_E(\xi)\,d\xi
}_{\Ccal(E)}.
\end{aligned}
\]
\end{proof}
\begin{proof}[Proof of Theorem~\ref{thm:main-stationary}]
By Lemma~\ref{lem:main-stationary-box},
$\displaystyle 
\Phi([0,\bar M])\subset[0,\bar M]$.
It remains to verify continuity. Fix \(E_n\to E\) in \([0,\bar M]\). Let
\[
q_n(l):=\frac{\mu(E_n,l)+u(l)}{g(E_n,l)},
\qquad
q(l):=\frac{\mu(E,l)+u(l)}{g(E,l)} .
\]
From
$g(E_n,\cdot)\to g(E,\cdot)$,
$\mu(E_n,\cdot)\to\mu(E,\cdot)$ in $L^\infty(\Om)$,
and
$g(E_n,l)\ge g_\flat>0$,
we obtain
$\displaystyle 
\norm{q_n-q}_{L^\infty(\Om)}\to0$.
Also
$\displaystyle 
\left\|
\frac{1}{g(E_n,\cdot)}-\frac{1}{g(E,\cdot)}
\right\|_{\infty}
\le
g_\flat^{-2}
\norm{g(E_n,\cdot)-g(E,\cdot)}_{\infty}
\to0$.
Now
\[
\begin{aligned}
x_{E_n}(l)-x_E(l)
&=
p\left[
\frac{1}{g(E_n,l)}e^{-\int_{l_0}^{l}q_n(\xi)\,d\xi}
-
\frac{1}{g(E,l)}e^{-\int_{l_0}^{l}q(\xi)\,d\xi}
\right] \\[4pt]
&=
p\left(\frac{1}{g(E_n,l)}-\frac{1}{g(E,l)}\right)
e^{-\int_{l_0}^{l}q_n\,d\xi}
+
\frac{p}{g(E,l)}
\left(
e^{-\int_{l_0}^{l}q_n\,d\xi}
-
e^{-\int_{l_0}^{l}q\,d\xi}
\right).
\end{aligned}
\]
Since \(q_n,q\ge0\), $0<e^{-\int_{l_0}^{l}q_n}\le1$ and 
$0<e^{-\int_{l_0}^{l}q}\le1$.
Moreover, using \(|e^{-a}-e^{-b}|\le |a-b|\) for \(a,b\ge0\),
\[
\left|
e^{-\int_{l_0}^{l}q_n}
-
e^{-\int_{l_0}^{l}q}
\right|
\le
\int_{l_0}^{l}\abs{q_n(\xi)-q(\xi)}\,d\xi
\le
(l_m-l_0)\norm{q_n-q}_{\infty}.
\]
Thus
\[
\norm{x_{E_n}-x_E}_{L^1}
\le
p(l_m-l_0)
g_\flat^{-2}
\norm{g(E_n,\cdot)-g(E,\cdot)}_{\infty}
+
p g_\flat^{-1}(l_m-l_0)^2\norm{q_n-q}_{\infty}
\to0.
\]
Therefore
$\displaystyle
\Phi(E_n)-\Phi(E)
=
\int_{l_0}^{l_m}\chi(l)(x_{E_n}(l)-x_E(l))\,dl$
satisfies
\[
\abs{\Phi(E_n)-\Phi(E)}
\le
\norm{\chi}_{\infty}\norm{x_{E_n}-x_E}_{L^1}
\to0.
\]
Hence
$\displaystyle 
\Phi\in C([0,\bar M])$.
Since \([0,\bar M]\) is compact and convex and \(\Phi\) is a continuous self-map,
there exists \(E^*\in[0,\bar M]\) such that
$\displaystyle 
E^*=\Phi(E^*)$.
For uniqueness, assume for all $E\in[0,\bar M]$,
$\Phi\in C^1([0,\bar M])$ and
$\Phi'(E)<1$.
Let
$H(E):=E-\Phi(E)$.
Then $H'(E)=1-\Phi'(E)>0$ on $[0,\bar M]$. 
Therefore \(H\) is strictly increasing and has at most one zero. 
Since existence is already proved, the stationary level is unique.
\end{proof}
\begin{proof}[Proof of Corollary~\ref{cor:main-dominance}]
Assume for all $E\in[0,\bar M]$,
$\Ccal(E)>\Rcal(E)$.
By Proposition~\ref{prop:main-decomp},
$\displaystyle 
\Phi'(E)=\Rcal(E)-\Ccal(E)<0$.
Therefore $\Phi'(E)<0<1$ on $[0,\bar M]$.
Theorem~\ref{thm:main-stationary} applies and gives uniqueness of the stationary
feedback level.
\end{proof}
\subsection{Proof of the fold theorem}
\label{subsec:proof-fold}
\begin{proof}[Proof of Theorem~\ref{thm:main-fold}]
Let
$H(E,\eta):=E-\Phi(E;\eta)$,
$H\in C^2(\mathcal I_E\times\mathcal I_\eta)$,
and assume
\[
H(E_0,\eta_0)=0,
\qquad
H_E(E_0,\eta_0)=0,
\qquad
H_{EE}(E_0,\eta_0)\ne0,
\qquad
H_\eta(E_0,\eta_0)\ne0 .
\]
Since
$\displaystyle 
H_\eta(E_0,\eta_0)\ne0$,
the implicit function theorem applied to
$\displaystyle 
H(E,\eta)=0$
in the \(\eta\)-variable gives neighborhoods
$E_0\in I_E$ and $\eta_0\in I_\eta$,
and a unique function
$\displaystyle 
\eta=\eta(E)\in C^2(I_E;I_\eta)$
such that
\[
H(E,\eta(E))=0,
\qquad
\eta(E_0)=\eta_0.
\]
Set
$s:=E-E_0$ and
$\zeta(s):=\eta(E_0+s)-\eta_0$.
Then
$\zeta(0)=0$ and
$H(E_0+s,\eta_0+\zeta(s))=0$.
Differentiate \(H(E,\eta(E))=0\) once:
\[
0
=
H_E(E,\eta(E))+H_\eta(E,\eta(E))\eta'(E).
\]
At \(E=E_0\),
\[
0
=
H_E(E_0,\eta_0)+H_\eta(E_0,\eta_0)\eta'(E_0)
=
0+H_\eta(E_0,\eta_0)\eta'(E_0).
\]
Since \(H_\eta(E_0,\eta_0)\ne0\),
$\displaystyle 
\eta'(E_0)=0$.
Equivalently,
$\displaystyle 
\zeta'(0)=0$.
Differentiate again:
\[
0
=
\frac{d}{dE}
\left[
H_E(E,\eta(E))+H_\eta(E,\eta(E))\eta'(E)
\right].
\]
Using the chain rule,
\[
\frac{d}{dE}H_E(E,\eta(E))
=
H_{EE}(E,\eta(E))
+
H_{E\eta}(E,\eta(E))\eta'(E),
\]
and
\[
\frac{d}{dE}
\big(H_\eta(E,\eta(E))\eta'(E)\big)
=
\big(H_{\eta E}(E,\eta(E))
+
H_{\eta\eta}(E,\eta(E))\eta'(E)\big)\eta'(E)
+
H_\eta(E,\eta(E))\eta''(E).
\]
Since \(H_{E\eta}=H_{\eta E}\),
\[
0
=
H_{EE}
+
2H_{E\eta}\eta'
+
H_{\eta\eta}(\eta')^2
+
H_\eta\eta''.
\]
At \(E=E_0\), using \(\eta'(E_0)=0\),
\[
0
=
H_{EE}(E_0,\eta_0)
+
H_\eta(E_0,\eta_0)\eta''(E_0).
\]
Thus
$\displaystyle 
\eta''(E_0)
=
-\frac{H_{EE}(E_0,\eta_0)}{H_\eta(E_0,\eta_0)}$.
Because
$H_{EE}(E_0,\eta_0)\ne0$,
$H_\eta(E_0,\eta_0)\ne0$,
we also have
$\displaystyle 
\eta''(E_0)\ne0$.
Taylor expansion of \(\eta(E)\) at \(E_0\) gives
\[
\eta(E_0+s)
=
\eta(E_0)
+
\eta'(E_0)s
+
\frac12\eta''(E_0)s^2
+
o(s^2).
\]
Using
$\eta(E_0)=\eta_0$,
$\eta'(E_0)=0$, and
$\eta''(E_0)
=
-\frac{H_{EE}(E_0,\eta_0)}{H_\eta(E_0,\eta_0)}$,
we obtain
\[
\eta(E_0+s)
=
\eta_0
-
\frac{H_{EE}(E_0,\eta_0)}
     {2H_\eta(E_0,\eta_0)}
s^2
+
o(s^2).
\]
Since
$\displaystyle 
H(E,\eta)=E-\Phi(E;\eta)$,
one has
\[
H_E=1-\Phi_E,
\qquad
H_\eta=-\Phi_\eta,
\qquad
H_{EE}=-\Phi_{EE}.
\]
The hypotheses may therefore be written as
\[
E_0=\Phi(E_0;\eta_0),
\qquad
\Phi_E(E_0;\eta_0)=1,
\qquad
\Phi_{EE}(E_0;\eta_0)\ne0,
\qquad
\Phi_\eta(E_0;\eta_0)\ne0.
\]
Moreover,
$\displaystyle 
-\frac{H_{EE}}{2H_\eta}
=
-\frac{\Phi_{EE}}{2\Phi_\eta}$,
all derivatives being evaluated at \((E_0,\eta_0)\). Thus
\[
\eta(E_0+s)
=
\eta_0
-
\frac{\Phi_{EE}(E_0;\eta_0)}
     {2\Phi_\eta(E_0;\eta_0)}
s^2
+
o(s^2).
\]
Equivalently, with
$\displaystyle 
\kappa_{\rm fold}
:=
-\frac{H_{EE}(E_0,\eta_0)}{2H_\eta(E_0,\eta_0)}
=
-\frac{\Phi_{EE}(E_0;\eta_0)}{2\Phi_\eta(E_0;\eta_0)}$,
the local zero set is
\[
\eta-\eta_0
=
\kappa_{\rm fold}(E-E_0)^2
+
o\!\left((E-E_0)^2\right).
\]
Since
$\displaystyle 
\kappa_{\rm fold}\ne0$,
the projection of the solution curve onto the \(\eta\)-axis has a quadratic
turning point at \((E_0,\eta_0)\).
If \(\kappa_{\rm fold}>0\), then for \(\eta>\eta_0\) sufficiently close to
\(\eta_0\),
$\displaystyle 
\eta-\eta_0
=
\kappa_{\rm fold}s^2+o(s^2)$
has two local branches
\[
s_\pm(\eta)
=
\pm
\sqrt{\frac{\eta-\eta_0}{\kappa_{\rm fold}}}
+
o\!\left(\sqrt{\abs{\eta-\eta_0}}\right),
\]
i.e.
\[
E_\pm(\eta)
=
E_0
\pm
\sqrt{\frac{\eta-\eta_0}{\kappa_{\rm fold}}}
+
o\!\left(\sqrt{\abs{\eta-\eta_0}}\right).
\]
For \(\eta<\eta_0\) sufficiently close to \(\eta_0\), no real \(s\ne0\) solves the
local quadratic equation. If \(\kappa_{\rm fold}<0\), the orientation is reversed:
two branches occur for \(\eta<\eta_0\), and none for \(\eta>\eta_0\).
Thus \((E_0,\eta_0)\) is a nondegenerate local fold of the stationary closure
equation.
\end{proof}
\subsection{Proof of the linearized renewal reduction}
\label{subsec:proof-renewal}
\begin{proof}[Proof of Lemma~\ref{lem:main-sweep}]
Let
\[
\tau(l):=\int_{l_0}^{l}\frac{d\xi}{g^*(\xi)},
\qquad
0=\tau(l_0)<\tau(l)\le \tau(l_m)=t_\sharp .
\]
Since \(g^*(l)\ge g_M>0\), the map
$\displaystyle 
l\mapsto \tau(l)$
is strictly increasing, absolutely continuous, and invertible on
\([0,t_\sharp]\). Let
$\displaystyle 
L(t;l)$
denote the forward characteristic issued from \(l\), i.e.
\[
\frac{d}{dt}L(t;l)=g^*(L(t;l)),
\qquad
L(0;l)=l .
\]
Equivalently,
$\displaystyle \tau(L(t;l))=\tau(l)+t$.
Hence the exit time from \(l\) is
$\displaystyle t_{\rm exit}(l)=t_\sharp-\tau(l)\le t_\sharp$.
Thus every characteristic issued from \(\Om\) reaches \(l_m\) not later than
\(t_\sharp\).

Let \(U_0(t)\) be the zero-inflow semigroup, and denote \(\psi(t, l) = U_0(t)\varphi(l)\). Recall that \(\psi\) satisfies the zero-inflow transport equation:
\[
\partial_t \psi + \partial_l (g^* \psi) = -(\mu^*+u)\psi.
\]
By applying the chain rule along the characteristic curve \(L(t;l)\) which satisfies \(\frac{d}{dt}L(t;l) = g^*(L(t;l))\), we evaluate the derivative:
\[
\frac{d}{dt}\big(g^*(L)\psi(t,L)\big) 
= (g^*)'g^*\psi + g^*(\partial_t \psi + g^*\partial_l \psi) 
= g^*\big[\partial_t \psi + \partial_l (g^*\psi)\big].
\]
Substituting the transport equation into the bracket, we arrive at the ordinary differential equation along characteristics:
\[
\frac{d}{dt}\big(g^*(L(t;l))\,U_0(t)\varphi(L(t;l))\big)
=
-(\mu^*+u)(L(t;l))\,\big[ g^*(L(t;l))\,U_0(t)\varphi(L(t;l)) \big].
\]
Therefore, before exit,
\[
g^*(L(t;l))\,U_0(t)\varphi(L(t;l))
=
g^*(l)\varphi(l)
\exp\!\left[
-\int_0^t(\mu^*+u)(L(s;l))\,ds
\right].
\]
After exit, the value is zero. Since zero inflow gives no re-entry at \(l_0\),
\[
t\ge t_\sharp
\qquad\Longrightarrow\qquad
U_0(t)\varphi\equiv0.
\]
Thus $t\ge t_\sharp$,
$U_0(t)=0$.
For the homogeneous eigenproblem,
\[
\lambda\varphi
=
-\partial_l(g^*\varphi)-(\mu^*+u)\varphi,
\qquad
g^*(l_0)\varphi(l_0)=0,
\]
set
$\displaystyle 
Y(l):=g^*(l)\varphi(l)$.
Then
$\displaystyle 
\lambda\frac{Y(l)}{g^*(l)}
=
-Y'(l)-(\mu^*(l)+u(l))\frac{Y(l)}{g^*(l)}$,
or
$\displaystyle 
Y'(l)
=
-\frac{\lambda+\mu^*(l)+u(l)}{g^*(l)}Y(l)$.
Hence
\[
Y(l)
=
Y(l_0)
\exp\!\left[
-\int_{l_0}^{l}
\frac{\lambda+\mu^*(\xi)+u(\xi)}{g^*(\xi)}\,d\xi
\right].
\]
Since
$\displaystyle 
Y(l_0)=g^*(l_0)\varphi(l_0)=0$,
we obtain
$Y\equiv0$ and
$\varphi\equiv0$.
The same characteristic formula with datum replaced by a measure source or by a
unit boundary-flux impulse gives
\[
U_0(t)\nu=0,
\qquad
\Sigma_b(t)=0,
\qquad
t\ge t_\sharp .
\]
\end{proof}
\begin{proof}[Proof of Proposition~\ref{prop:main-renewal}]
The linearized system is
\[
\partial_t v+\partial_l(g^*v)+(\mu^*+u)v
=
\sigma_{\rm int}e(t),
\qquad
e(t)=\langle\chi,v(t)\rangle,
\]
with boundary flux
$\displaystyle 
g^*(l_0)v(t,l_0)=-a_0e(t)$.
Let \(e\) first be an arbitrary scalar input. Decompose
$\displaystyle 
v=v^{\rm hom}+v^{\rm int}+v^{\rm bd}$.
The zero-inflow homogeneous part is
$\displaystyle 
v^{\rm hom}(t)=U_0(t)v_0$.
The interior-source part is
$\displaystyle 
v^{\rm int}(t)
=
\int_0^t U_0(t-s)\sigma_{\rm int}\,e(s)\,ds$.
Let \(\Sigma_b(\tau)\) denote the state at time \(\tau\) generated by unit inflow
flux at \(l_0\) at time \(0\). Since the actual linearized boundary flux is
$\displaystyle 
-a_0e(s)$,
the boundary contribution is
$\displaystyle 
v^{\rm bd}(t)
=
-a_0\int_0^t\Sigma_b(t-s)e(s)\,ds$.
Therefore
\[
v(t)
=
U_0(t)v_0
+
\int_0^t U_0(t-s)\sigma_{\rm int}\,e(s)\,ds
-
a_0\int_0^t \Sigma_b(t-s)e(s)\,ds .
\]
Pair with \(\chi\):
\[
\begin{aligned}
e(t)
&=
\langle\chi,U_0(t)v_0\rangle
+
\int_0^t
\langle\chi,U_0(t-s)\sigma_{\rm int}\rangle e(s)\,ds
-
a_0\int_0^t
\langle\chi,\Sigma_b(t-s)\rangle e(s)\,ds\\[4pt]
&=f(t)+\int_0^t\kappa(t-s)e(s)\,ds,
\end{aligned}
\]
where
$\displaystyle 
f(t):=\langle\chi,U_0(t)v_0\rangle$,
and
$\kappa(\tau)
:=
\langle\chi,U_0(\tau)\sigma_{\rm int}\rangle
-
a_0\langle\chi,\Sigma_b(\tau)\rangle$.
By Lemma~\ref{lem:main-sweep}, when $\tau\ge t_\sharp$, $U_0(\tau)=0$ and $\Sigma_b(\tau)=0$.
Consequently,
$\operatorname{supp}f\subset[0,t_\sharp]$ and
$\operatorname{supp}\kappa\subset[0,t_\sharp]$.
It remains only to record the regularity of \(\kappa\). Since
\[
\sigma_{\rm int}
=
-\frac{d}{dl}\big((\partial_E g)(E^*,\cdot)x^*\big)
-(\partial_E\mu)(E^*,\cdot)x^*
\in\Mm(\Om),
\]
and since \(\chi\in W^{1,\infty}(\Om)\), the map
$\tau\mapsto \langle\chi,U_0(\tau)\sigma_{\rm int}\rangle$
is a finite sum of absolutely continuous characteristic contributions plus
finitely many transported atom contributions; hence it is of bounded variation
on \([0,t_\sharp]\). Similarly,
$\tau\mapsto \langle\chi,\Sigma_b(\tau)\rangle$
is a characteristic trace response and belongs to \(BV([0,t_\sharp])\). Therefore
$\kappa\in BV([0,t_\sharp])$.
\end{proof}
\begin{proof}[Proof of Lemma~\ref{lem:main-Ev-props}]
Let \(\lambda\in\C\). The inhomogeneous stationary transport problem is
\[
\frac{d}{dl}\big(g^*\varphi_1^\lambda\big)
+
(\lambda+\mu^*+u)\varphi_1^\lambda
=
\sigma_{\rm int},
\qquad
g^*(l_0)\varphi_1^\lambda(l_0)=-a_0 .
\]
Set
\[
Y^\lambda(l):=g^*(l)\varphi_1^\lambda(l),
\qquad
Q^\lambda(l):=
\frac{\lambda+\mu^*(l)+u(l)}{g^*(l)}.
\]
Then
$\displaystyle 
(Y^\lambda)'(l)+Q^\lambda(l)Y^\lambda(l)=\sigma_{\rm int}$
in the distributional sense, and
$\displaystyle 
Y^\lambda(l_0)=-a_0$.
Define
\[
\Pi^\lambda(l)
:=
\exp\!\left[-\int_{l_0}^{l}Q^\lambda(\xi)\,d\xi\right]
=
\exp\!\left[-\int_{l_0}^{l}
\frac{\lambda+\mu^*(\xi)+u(\xi)}{g^*(\xi)}\,d\xi
\right].
\]
Then
$\displaystyle 
\frac{d}{dl}\left((\Pi^\lambda(l))^{-1}Y^\lambda(l)\right)
=
(\Pi^\lambda(l))^{-1}\,d\sigma_{\rm int}(l)$.
Therefore
\[
(\Pi^\lambda(l))^{-1}Y^\lambda(l)
=
Y^\lambda(l_0)
+
\int_{l_0}^{l}(\Pi^\lambda(s))^{-1}\,d\sigma_{\rm int}(s),
\]
and hence
\[
Y^\lambda(l)
=
\Pi^\lambda(l)
\left[
-a_0+\int_{l_0}^{l}(\Pi^\lambda(s))^{-1}\,d\sigma_{\rm int}(s)
\right].
\]
Since \(Y^\lambda=g^*\varphi_1^\lambda\),
\[
\varphi_1^\lambda(l)
=
\frac{\Pi^\lambda(l)}{g^*(l)}
\left[
-a_0+\int_{l_0}^{l}(\Pi^\lambda(s))^{-1}\,d\sigma_{\rm int}(s)
\right].
\]
Thus
\begin{equation}
\label{eq:proof-Ev-explicit}
\Ev(\lambda)
=
\int_{l_0}^{l_m}
\frac{\chi(l)}{g^*(l)}
\Pi^\lambda(l)
\left[
-a_0+\int_{l_0}^{l}(\Pi^\lambda(s))^{-1}\,d\sigma_{\rm int}(s)
\right]dl .
\end{equation}
For each \(l\),
$\displaystyle 
\Pi^\lambda(l)
=
\exp\!\left[-\lambda\int_{l_0}^{l}\frac{d\xi}{g^*(\xi)}\right]
\Pi^0(l)$,
hence \(\lambda\mapsto\Pi^\lambda(l)\) is entire. Likewise,
$\displaystyle 
(\Pi^\lambda(s))^{-1}
=
\exp\!\left[\lambda\int_{l_0}^{s}\frac{d\xi}{g^*(\xi)}\right]
(\Pi^0(s))^{-1}$
is entire. For \(\lambda\) in a compact set \(K\subset\C\),
\[
\abs{\Pi^\lambda(l)}
\le
e^{C_Kt_\sharp}\Pi^0(l),
\qquad
\abs{(\Pi^\lambda(s))^{-1}}
\le
e^{C_Kt_\sharp}(\Pi^0(s))^{-1},
\]
where
$\displaystyle 
C_K:=\sup_{\lambda\in K}\abs{\lambda}$.
Since \(\sigma_{\rm int}\in\Mm(\Om)\), the integrand in
\eqref{eq:proof-Ev-explicit} is locally dominated uniformly in \(\lambda\). Hence
$\Ev: \mathbb C\mapsto \mathbb C$ is entire.
Because all coefficients and measures are real,
\[
\overline{\Pi^\lambda(l)}=\Pi^{\bar\lambda}(l),
\qquad
\overline{\int_{l_0}^{l}(\Pi^\lambda(s))^{-1}\,d\sigma_{\rm int}(s)}
=
\int_{l_0}^{l}(\Pi^{\bar\lambda}(s))^{-1}\,d\sigma_{\rm int}(s),
\]
so
$\displaystyle 
\Ev(\bar\lambda)=\overline{\Ev(\lambda)}$.
Next, compare \(\Ev\) with the Laplace transform of the renewal kernel. Let
$v(t,l)=e^{\lambda t}\varphi(l)$ and 
$e(t)=e^{\lambda t}\hat e$.
Substitution into the linearized system gives
\[
\frac{d}{dl}(g^*\varphi)+(\lambda+\mu^*+u)\varphi
=
\sigma_{\rm int}\hat e,
\qquad
g^*(l_0)\varphi(l_0)=-a_0\hat e .
\]
Thus
$\displaystyle 
\varphi=\hat e\,\varphi_1^\lambda$.
The consistency condition \(e=\langle\chi,v\rangle\) gives
$\displaystyle 
\hat e
=
\langle\chi,\varphi\rangle
=
\hat e\,\langle\chi,\varphi_1^\lambda\rangle
=
\hat e\,\Ev(\lambda)$.
On the other hand, for a homogeneous modal solution of the scalar renewal
equation,
\[
e(t)=e^{\lambda t}\hat e,
\qquad
f=0,
\]
we have
$\displaystyle 
e^{\lambda t}\hat e
=
\int_0^t\kappa(t-s)e^{\lambda s}\hat e\,ds$.
Writing \(\tau=t-s\),
$\displaystyle 
e^{\lambda t}\hat e
=
e^{\lambda t}\hat e\int_0^t e^{-\lambda\tau}\kappa(\tau)\,d\tau$.
Since \(\operatorname{supp}\kappa\subset[0,t_\sharp]\), for \(t\ge t_\sharp\),
$\displaystyle 
\hat e
=
\hat e\int_0^{t_\sharp}e^{-\lambda\tau}\kappa(\tau)\,d\tau$.
Therefore the two modal consistency factors coincide:
\[
\Ev(\lambda)
=
\int_0^{t_\sharp}e^{-\lambda\tau}\kappa(\tau)\,d\tau.
\]
Now let \(\lambda\to+\infty\), \(\lambda\in\R\). In \eqref{eq:proof-Ev-explicit},
\[
\Pi^\lambda(l)
=
e^{-\lambda\tau(l)}\Pi^0(l),
\qquad
\frac{\Pi^\lambda(l)}{\Pi^\lambda(s)}
=
e^{-\lambda(\tau(l)-\tau(s))}
\frac{\Pi^0(l)}{\Pi^0(s)}.
\]
For \(l>l_0\),
$\displaystyle 
e^{-\lambda\tau(l)}\to0$.
For \(s<l\),
$\displaystyle  e^{-\lambda(\tau(l)-\tau(s))}\to0$.
The diagonal \(s=l\) is irrelevant for the absolutely continuous part and is
controlled for atoms by the one-sided Stieltjes convention in the variation-of-constants formula. The dominated bounds
\[
0\le e^{-\lambda\tau(l)}\le1,
\qquad
0\le e^{-\lambda(\tau(l)-\tau(s))}\le1
\]
give
$\displaystyle 
\Ev(\lambda)\to0$.
For \(\Re\lambda\ge0\),
$\displaystyle 
\abs{\Pi^\lambda(l)}
=
e^{-\Re\lambda\,\tau(l)}\Pi^0(l)
\le
\Pi^0(l)$,
and for \(s\le l\),
$\displaystyle 
\frac{\abs{\Pi^\lambda(l)}}{\abs{\Pi^\lambda(s)}}
=
e^{-\Re\lambda(\tau(l)-\tau(s))}
\frac{\Pi^0(l)}{\Pi^0(s)}
\le
\frac{\Pi^0(l)}{\Pi^0(s)}$.
Therefore, from \eqref{eq:proof-Ev-explicit},
\[
\abs{\Ev(\lambda)}
\le
\int_{l_0}^{l_m}
\frac{\chi(l)}{g^*(l)}
\Pi^0(l)
\left[
\abs{a_0}
+
\int_{l_0}^{l}(\Pi^0(s))^{-1}\,d\abs{\sigma_{\rm int}}(s)
\right]dl
=
G_0.
\]
Finally, put \(\lambda=0\). The equation for \(\varphi_1^0\) is
\[
\frac{d}{dl}(g^*\varphi_1^0)+(\mu^*+u)\varphi_1^0
=
\sigma_{\rm int},
\qquad
g^*(l_0)\varphi_1^0(l_0)=-a_0.
\]
The forward sensitivity
$\displaystyle 
y(l):=\partial_E x_E(l)|_{E=E^*}$
satisfies exactly
\[
\frac{d}{dl}(g^*y)+(\mu^*+u)y
=
\sigma_{\rm int},
\qquad
g^*(l_0)y(l_0)=-a_0 .
\]
Uniqueness of the first-order boundary-value problem gives
$\displaystyle 
\varphi_1^0=y$.
Therefore
$\displaystyle 
\Ev(0)
=
\langle\chi,\varphi_1^0\rangle
=
\langle\chi,y\rangle
=
\Phi'(E^*)$.
\end{proof}
\subsection{Proof of the spectral characterization}
\label{subsec:proof-spectrum}
\begin{proof}[Proof of Theorem~\ref{thm:main-spectrum}]
Let
\[
v(t,l)=e^{\lambda t}\varphi(l),
\qquad
e(t)=\langle\chi,v(t,\cdot)\rangle=e^{\lambda t}\hat e .
\]
Substitution into the linearized system gives
\[
\lambda\varphi+\partial_l(g^*\varphi)+(\mu^*+u)\varphi
=
\sigma_{\rm int}\hat e,
\qquad
g^*(l_0)\varphi(l_0)=-a_0\hat e .
\]
If \(\hat e\ne0\), division by \(\hat e\) gives
\[
\frac{d}{dl}\left(g^*\frac{\varphi}{\hat e}\right)
+
(\lambda+\mu^*+u)\frac{\varphi}{\hat e}
=
\sigma_{\rm int},
\qquad
g^*(l_0)\frac{\varphi(l_0)}{\hat e}=-a_0 .
\]
By Definition~\ref{def:main-characteristic},
$\frac{\varphi}{\hat e}=\varphi_1^\lambda$ and
$\varphi=\hat e\,\varphi_1^\lambda$.
The scalar consistency condition gives
$\displaystyle 
\hat e
=
\langle\chi,\varphi\rangle
=
\hat e\,\langle\chi,\varphi_1^\lambda\rangle
=
\hat e\,\Ev(\lambda)$.
Since \(\hat e\ne0\),
$\Ev(\lambda)=1$.
Conversely, if
$\displaystyle 
\Ev(\lambda)=1$,
then choose
$\hat e:=1$ and
$\varphi:=\varphi_1^\lambda$.
Then
\[
\frac{d}{dl}(g^*\varphi)+(\lambda+\mu^*+u)\varphi=\sigma_{\rm int},
\qquad
g^*(l_0)\varphi(l_0)=-a_0,
\]
and
$\displaystyle 
\langle\chi,\varphi\rangle
=
\langle\chi,\varphi_1^\lambda\rangle
=
\Ev(\lambda)=1=\hat e$.
Thus
\[
v(t,l):=e^{\lambda t}\varphi_1^\lambda(l),
\qquad
e(t):=e^{\lambda t}
\]
is a nontrivial feedback eigenmode.
If \(\hat e=0\), then the modal equation reduces to
\[
\frac{d}{dl}(g^*\varphi)+(\lambda+\mu^*+u)\varphi=0,
\qquad
g^*(l_0)\varphi(l_0)=0 .
\]
By Lemma~\ref{lem:main-sweep},
$\displaystyle 
\varphi\equiv0$.
Hence no nontrivial eigenmode comes from the zero-inflow homogeneous transport
part. Therefore the nontrivial feedback point spectrum is exactly
$\displaystyle 
\sigma_{\rm fb}
=
\{\lambda\in\C:\Ev(\lambda)=1\}$.
By Lemma~\ref{lem:main-Ev-props},
\[
\Ev \text{ is entire},
\qquad
\Ev(\lambda)\to0
\qquad
(\lambda\to+\infty,\ \lambda\in\R).
\]
Hence
$\displaystyle 
F(\lambda):=\Ev(\lambda)-1$
is entire and not identically zero. Therefore its zeros are isolated:
$\displaystyle 
F^{-1}(0)=\{\lambda\in\C:\Ev(\lambda)=1\}$
is a discrete subset of \(\C\).
Moreover,
\[
\Ev(\lambda)
=
\int_0^{t_\sharp}e^{-\lambda\tau}\kappa(\tau)\,d\tau,
\qquad
\kappa\in BV([0,t_\sharp]).
\]
For \(\Re\lambda\ge -\omega\),
$\displaystyle 
\abs{\Ev(\lambda)}
\le
\int_0^{t_\sharp}e^{-\Re\lambda\,\tau}\abs{\kappa(\tau)}\,d\tau
\le
e^{\omega t_\sharp}\norm{\kappa}_{L^1(0,t_\sharp)}$.
For large positive real part,
\[
\Re\lambda\to+\infty
\qquad\Longrightarrow\qquad
\Ev(\lambda)\to0.
\]
Since \(\kappa\in BV([0,t_\sharp])\), integration by parts gives, for
\(\lambda\ne0\),
\[
\Ev(\lambda)
=
\int_0^{t_\sharp}e^{-\lambda\tau}\kappa(\tau)\,d\tau
=
\frac{\kappa(0)-e^{-\lambda t_\sharp}\kappa(t_\sharp)}{\lambda}
+
\frac{1}{\lambda}
\int_{[0,t_\sharp]}e^{-\lambda\tau}\,d\kappa(\tau).
\]
Hence, for every \(\omega>0\), there is a constant \(C_\omega\) with
\[
\abs{\Ev(\lambda)}
\le
\frac{C_\omega}{\abs{\lambda}}
\qquad
\text{for }\Re\lambda\ge-\omega ,
\]
so there exists \(R_\omega>0\) such that
\[
\abs{\lambda}\ge R_\omega,\qquad \Re\lambda\ge-\omega
\qquad\Longrightarrow\qquad
\abs{\Ev(\lambda)}<1/2.
\]
Thus every root of \(\Ev(\lambda)=1\) in \(\{\Re\lambda\ge-\omega\}\) lies in the
compact set
$\displaystyle 
\{\Re\lambda\ge-\omega,\ \abs{\lambda}\le R_\omega\}$,
and, since \(\Ev-1\) is entire and not identically zero, this set contains only
finitely many roots. This is the standard root-localization for renewal and
retarded characteristic equations \cite{diekmann2012delay,gripenberg1990volterra,yu2026size}.
Now consider the scalar renewal equation
\[
e=f+\kappa*e,
\qquad
(\kappa*e)(t):=\int_0^t\kappa(t-s)e(s)\,ds .
\]
Let
$\displaystyle 
\widehat h(\lambda):=\int_0^\infty e^{-\lambda t}h(t)\,dt$.
For \(\Re\lambda\) sufficiently large,
\[
\widehat e(\lambda)
=
\widehat f(\lambda)
+
\widehat\kappa(\lambda)\widehat e(\lambda),
\qquad
\widehat\kappa(\lambda)
=
\int_0^{t_\sharp}e^{-\lambda\tau}\kappa(\tau)\,d\tau
=
\Ev(\lambda).
\]
Thus
$\displaystyle 
(1-\Ev(\lambda))\widehat e(\lambda)=\widehat f(\lambda)$.
The Volterra resolvent \(\rho\) is defined by
$\displaystyle 
\rho=\kappa+\kappa*\rho$,
so, after Laplace transform,
$\displaystyle 
\widehat\rho(\lambda)
=
\widehat\kappa(\lambda)
+
\widehat\kappa(\lambda)\widehat\rho(\lambda)$,
hence
$\displaystyle 
(1-\widehat\kappa(\lambda))\widehat\rho(\lambda)
=
\widehat\kappa(\lambda)$,
and therefore
$\displaystyle 
\widehat\rho(\lambda)
=
\frac{\widehat\kappa(\lambda)}{1-\widehat\kappa(\lambda)}
=
\frac{\Ev(\lambda)}{1-\Ev(\lambda)}$.
Set
$\omega_0:=\sup\{\Re\lambda:\Ev(\lambda)=1\}$.
For every \(\omega>\omega_0\), any $\lambda$ such that $\Re\lambda\ge\omega$ satisfies $1-\Ev(\lambda)\ne0$.
The Paley--Wiener/Volterra resolvent criterion for convolution kernels
\cite{gripenberg1990volterra} gives
\[
\rho\in e^{\omega\cdot}L^1(0,\infty),
\qquad
\norm{\rho}_{L^1_\omega}:=
\int_0^\infty e^{-\omega t}\abs{\rho(t)}\,dt<\infty .
\]
Thus the solution representation
$\displaystyle 
e=f+\rho*f$
satisfies
$\displaystyle 
\norm{e}_\omega
\le
\big(1+\norm{\rho}_{L^1_\omega}\big)\norm{f}_\omega$.
For \(\omega<\omega_0\), there exists a root \(\lambda_*\) with
$\Ev(\lambda_*)=1$,
$\Re\lambda_*>\omega$,
and hence the modal solution
$\displaystyle 
e(t)=e^{\lambda_*t}$
precludes decay with weight \(e^{-\omega t}\). Therefore the scalar renewal
feedback growth bound is
$\omega_0$.
Consequently,
\[
\omega_0<0
\qquad\Longrightarrow\qquad
\text{linearized exponential stability},
\]
whereas
\[
\omega_0>0
\qquad\Longrightarrow\qquad
\text{existence of an exponentially growing feedback mode}.
\]
\end{proof}
\begin{proof}[Proof of Theorem~\ref{thm:main-instability}]
By Lemma~\ref{lem:main-Ev-props},
$\displaystyle 
\Ev(0)=\Phi'(E^*)$.
Assume
$\displaystyle 
\Phi'(E^*)>1$.
Then
$\displaystyle 
\Ev(0)>1$.
Also, along the positive real axis,
$\Ev(\lambda)\to0$ as $\lambda\to+\infty$ for $\lambda\in\R$.
Since \(\Ev(\lambda)\in\R\) for \(\lambda\in\R\), the map
$\displaystyle 
[0,\infty)\ni\lambda\mapsto \Ev(\lambda)$
is continuous and real-valued. Choose \(R>0\) such that
$\displaystyle 
\Ev(R)<1$.
Then
\[
\Ev(0)-1>0,
\qquad
\Ev(R)-1<0.
\]
By the intermediate value theorem, there exists
$\displaystyle 
\lambda^*\in(0,R)$
with
\[
\Ev(\lambda^*)-1=0,
\qquad
\Ev(\lambda^*)=1.
\]
Hence
$\lambda^*>0$,
$\lambda^*\in\sigma_{\rm fb}$,
and
$\omega_0\ge\lambda^*>0$.
Theorem~\ref{thm:main-spectrum} gives linearized instability.
\end{proof}
\begin{proof}[Proof of Theorem~\ref{thm:main-gain-stability}]
Assume
$\displaystyle 
G_0<1$.
By Lemma~\ref{lem:main-Ev-props}, for every \(\lambda\) with \(\Re\lambda\ge0\),
$\abs{\Ev(\lambda)}\le G_0<1$.
Therefore
\[
\Ev(\lambda)\ne1
\qquad
(\Re\lambda\ge0).
\]
Thus
$\displaystyle 
\sup\{\Re\lambda:\Ev(\lambda)=1\}<0$,
i.e.
$\displaystyle 
\omega_0<0$.
Theorem~\ref{thm:main-spectrum} yields linearized exponential stability.
At \(\lambda=0\),
$\displaystyle 
\abs{\Phi'(E^*)}
=
\abs{\Ev(0)}
\le
G_0<1$.
Thus
\[
G_0<1
\qquad\Longrightarrow\qquad
\abs{\Phi'(E^*)}<1.
\]
If the nonlinear quadratic remainder estimate
$\displaystyle 
\norm{\Gcal[e,v_0]}_\omega
\le
C\big(\norm{v_0}_{BV}+\norm{e}_\omega\big)^2$
from Theorem~\ref{thm:main-nonlinear-stability} holds, then the linear decay
$\displaystyle 
\omega_0<0$
and the nonlinear renewal equation
$\displaystyle 
e=f+\kappa*e+\Gcal[e,v_0]$
give local nonlinear exponential stability by
Theorem~\ref{thm:main-nonlinear-stability}.
\end{proof}
\begin{proof}[Proof of Theorem~\ref{thm:main-nonlinear-stability}]
The detailed nonlinear verification is given in Appendix~\ref{app:nonlinear}.
Here we record the dependency chain.
The nonlinear scalar perturbation satisfies
$\displaystyle 
e=f+\kappa*e+\Gcal[e,v_0]$,
where the linear resolvent kernel \(\rho\) satisfies
$\displaystyle 
\rho=\kappa+\kappa*\rho$.
Hence
$e
=
f+\rho*f+\Gcal[e,v_0]+\rho*\Gcal[e,v_0]$.
Let
\[
\omega\in(\omega_0,0),
\qquad
r_\omega:=\norm{\rho}_{L^1_\omega}<\infty.
\]
Then
$\displaystyle 
\norm{e}_\omega
\le
(1+r_\omega)\norm{f}_\omega
+
(1+r_\omega)\norm{\Gcal[e,v_0]}_\omega$.
The free term satisfies
$\displaystyle 
\norm{f}_\omega
\le
C_f\norm{v_0}_{BV}$.
The quadratic hypothesis gives
$\displaystyle 
\norm{\Gcal[e,v_0]}_\omega
\le
C\big(\norm{v_0}_{BV}+\norm{e}_\omega\big)^2$.
Therefore
\[
\norm{e}_\omega
\le
(1+r_\omega)C_f\norm{v_0}_{BV}
+
(1+r_\omega)C\big(\norm{v_0}_{BV}+\norm{e}_\omega\big)^2.
\]
For
$\norm{v_0}_{BV}\le\delta$ with $\delta>0$ sufficiently small,
the standard quadratic bootstrap gives
$\displaystyle 
\norm{e}_\omega
\le
C_e\norm{v_0}_{BV}$.
The nilpotent-window estimate gives
\[
\norm{x(t)-x^*}_{L^1}
\le
C_\sharp\int_{(t-T_\sharp)_+}^{t}\abs{e(s)}\,ds
+
\mathbf 1_{\{t<T_\sharp\}}\norm{v_0}_{L^1}.
\]
Since
$\displaystyle 
\abs{e(s)}
\le
\norm{e}_\omega e^{\omega s}
\le
C_e\norm{v_0}_{BV}e^{\omega s}$,
we get, for \(t\ge T_\sharp\),
\[
\norm{x(t)-x^*}_{L^1}
\le
C_\sharp C_e\norm{v_0}_{BV}
\int_{t-T_\sharp}^{t}e^{\omega s}\,ds
\le
C_\sharp C_e T_\sharp e^{|\omega|T_\sharp}
e^{\omega t}\norm{v_0}_{BV}.
\]
For \(0\le t<T_\sharp\),
$\displaystyle 
\norm{x(t)-x^*}_{L^1}
\le
C\norm{v_0}_{BV}
\le
Ce^{|\omega|T_\sharp}e^{\omega t}\norm{v_0}_{BV}$.
Finally,
\[
\abs{E(t)-E^*}
=
\abs{\langle\chi,x(t)-x^*\rangle}
\le
\norm{\chi}_{\infty}\norm{x(t)-x^*}_{L^1}.
\]
Combining the estimates gives
\[
\norm{x(t)-x^*}_{L^1}+\abs{E(t)-E^*}
\le
M e^{\omega t}\norm{v_0}_{BV},
\qquad
t\ge0.
\]
Under \eqref{eq:model-C2E}, Appendix~\ref{app:nonlinear} proves the quadratic
estimate for \(\Gcal\), completing the verification.
\end{proof}
\subsection{Proof of the rank-one adjoint formula}
\label{subsec:proof-adjoint}
\begin{proof}[Proof of Lemma~\ref{lem:main-loop-admissibility}]
Let
$f\in L^1(\Om)$ and
$\lambda:=R_rf$.
Then
$\Lcal_r^*\lambda=f$ and
$\lambda(l_m)=0$,
that is,
\[
-g^*(l)\lambda'(l)
+
\big(r+\mu^*(l)+u(l)\big)\lambda(l)=f(l),
\qquad
\lambda(l_m)=0 .
\]
Set
\[
Q_r(l):=\frac{r+\mu^*(l)+u(l)}{g^*(l)},
\qquad
\Pi_r(l,s):=
\exp\!\left[
-\int_l^s Q_r(\xi)\,d\xi
\right],
\qquad
l\le s.
\]
Since
$r\ge0$,
$\mu^*\ge0$,
$u\ge0$, and
$g^*\ge g_M>0$,
one has
$Q_r\ge0$ and
$0<\Pi_r(l,s)\le1$.
Solving the terminal-value equation gives
\[
\lambda(l)
=
(R_rf)(l)
=
\int_l^{l_m}
\frac{f(s)}{g^*(s)}
\Pi_r(l,s)\,ds.
\]
Hence
$\displaystyle 
\abs{\lambda(l)}
\le
\int_l^{l_m}
\frac{\abs{f(s)}}{g^*(s)}
\Pi_r(l,s)\,ds
\le
g_M^{-1}\int_l^{l_m}\abs{f(s)}\,ds
\le
g_M^{-1}\norm{f}_{L^1}$.
Therefore
\[
\norm{R_rf}_{L^\infty(\Om)}
\le
g_M^{-1}\norm{f}_{L^1(\Om)}.
\]
Also,
\begin{equation} 
\label{eq:proof-adjoint-Linf}
\norm{R_rf}_{L^1(\Om)}
\le
(l_m-l_0)\norm{R_rf}_{L^\infty(\Om)}
\le
(l_m-l_0)g_M^{-1}\norm{f}_{L^1(\Om)}.
\end{equation}
From
$\displaystyle 
-g^*\lambda'
+
(r+\mu^*+u)\lambda
=f$,
we obtain
$\displaystyle 
\lambda'
=
-\frac{f}{g^*}
+
\frac{r+\mu^*+u}{g^*}\lambda$.
Thus
\[
\abs{\lambda'}
\le
g_M^{-1}\abs{f}
+
g_M^{-1}(r+\norm{\mu^*}_{\infty}+u_{\max})\abs{\lambda}
\le g_M^{-1}\abs{f}
+
g_M^{-1}(r+C_M+u_{\max})\abs{\lambda}.
\]
Integrating over \(\Om\),
\[
\norm{\lambda'}_{L^1}
\le
g_M^{-1}\norm{f}_{L^1}
+
g_M^{-1}(r+C_M+u_{\max})\norm{\lambda}_{L^1}.
\]
Using \eqref{eq:proof-adjoint-Linf},
\[
\norm{\lambda'}_{L^1}
\le
g_M^{-1}\norm{f}_{L^1}
+
g_M^{-1}(r+C_M+u_{\max})(l_m-l_0)g_M^{-1}\norm{f}_{L^1}.
\]
Hence
\begin{equation}
\label{eq:proof-adjoint-W11}
\norm{(R_rf)'}_{L^1}
\le
g_M^{-1}
\left[
1+(r+C_M+u_{\max})(l_m-l_0)g_M^{-1}
\right]\norm{f}_{L^1}.
\end{equation}
Therefore
$R_rf\in W^{1,1}(\Om)$ and $(R_rf)(l_m)=0$, and $R_rf\in X_*$.
Now use the corrected co-load pairing. For every \(\lambda\in X_*\),
\[
\langle\sigma,\lambda\rangle
=
\int_{\Om}x^*(l)
\Big[
(\partial_Eg)(E^*,l)\lambda'(l)
-
(\partial_E\mu)(E^*,l)\lambda(l)
\Big]\,dl .
\]
By
$\abs{\partial_Eg(E^*,l)}\le C_M$ and
$\abs{\partial_E\mu(E^*,l)}\le C_M$,
we get
$\displaystyle 
\abs{\langle\sigma,\lambda\rangle}
\le
C_M\norm{x^*}_{L^\infty}
\left(
\norm{\lambda'}_{L^1}
+
\norm{\lambda}_{L^1}
\right)$.
With \(\lambda=R_rf\),
\[
\abs{\langle\sigma,R_rf\rangle}
\le
C_M\norm{x^*}_{L^\infty}
\left[
\norm{(R_rf)'}_{L^1}
+
\norm{R_rf}_{L^1}
\right].
\]
Using \eqref{eq:proof-adjoint-Linf}--\eqref{eq:proof-adjoint-W11},
$\displaystyle 
\abs{\langle\sigma,R_rf\rangle}
\le
C_{\sigma,r}\norm{f}_{L^1}$,
where
\[
C_{\sigma,r}
:=
C_M\norm{x^*}_{L^\infty}
\left\{
g_M^{-1}
\left[
1+(r+C_M+u_{\max})(l_m-l_0)g_M^{-1}
\right]
+
(l_m-l_0)g_M^{-1}
\right\}.
\]
Thus
$\displaystyle 
f\mapsto\langle\sigma,R_rf\rangle$
is bounded on \(L^1(\Om)\). In particular,
$\displaystyle 
B(r)=\langle\sigma,R_r\chi\rangle$
is well-defined whenever \(\chi\in L^1(\Om)\cap L^\infty(\Om)\).
\end{proof}
\begin{proof}[Proof of Theorem~\ref{thm:main-rankone}]
Let
\[
\mathcal A_r:=\Lcal_r^*-\chi\langle\sigma,\cdot\rangle,
\qquad
\psi_r:=R_r\chi,
\qquad
\lambda_q:=R_rq,
\qquad
B(r):=\langle\sigma,\psi_r\rangle .
\]
The adjoint equation is
$\displaystyle 
\mathcal A_r\lambda=q$,
that is,
$\displaystyle 
\Lcal_r^*\lambda-\chi\langle\sigma,\lambda\rangle=q$.
Equivalently,
$\displaystyle 
\Lcal_r^*\lambda
=
q+\chi\langle\sigma,\lambda\rangle$.
Apply \(R_r=(\Lcal_r^*)^{-1}\):
$\lambda
=
R_rq
+
R_r\chi\,\langle\sigma,\lambda\rangle$.
Thus
\begin{equation}
\label{eq:proof-rankone-lambda-scalar}
\lambda
=
\lambda_q+\psi_r c,
\qquad
c:=\langle\sigma,\lambda\rangle .
\end{equation}
Pairing \eqref{eq:proof-rankone-lambda-scalar} with \(\sigma\),
$\displaystyle 
c
=
\langle\sigma,\lambda_q\rangle
+
\langle\sigma,\psi_r\rangle c
=
\langle\sigma,\lambda_q\rangle+B(r)c$.
Hence
\begin{equation}
\label{eq:proof-rankone-scalar-equation}
(1-B(r))c=\langle\sigma,\lambda_q\rangle .
\end{equation}
Assume first
$\displaystyle 
1-B(r)\ne0$.
Then
$\displaystyle 
c
=
\frac{\langle\sigma,\lambda_q\rangle}{1-B(r)}$.
Substitution into \eqref{eq:proof-rankone-lambda-scalar} yields
$\displaystyle 
\lambda
=
\lambda_q
+
\frac{\langle\sigma,\lambda_q\rangle}{1-B(r)}\psi_r$.
Equivalently,
\begin{equation}
\label{eq:proof-rankone-inverse}
\mathcal A_r^{-1}q
=
R_rq
+
\frac{\langle\sigma,R_rq\rangle}{1-\langle\sigma,R_r\chi\rangle}R_r\chi .
\end{equation}
We show uniqueness. If $\mathcal A_r\lambda=0$,
then \(q=0\), hence \(\lambda_q=0\), and
$\lambda=\psi_r c$. Pairing gives
$c=B(r)c$ and $(1-B(r))c=0$.
Since \(1-B(r)\ne0\),
$c=0$ and therefore $\lambda=0$.
Therefore the solution \eqref{eq:proof-rankone-inverse} is unique.
It remains to verify that \eqref{eq:proof-rankone-inverse} indeed solves the
equation. Let
$\displaystyle 
\lambda
=
\lambda_q
+
\frac{\langle\sigma,\lambda_q\rangle}{1-B(r)}\psi_r$.
Then $\Lcal_r^*\lambda_q=q$ and $\Lcal_r^*\psi_r=\chi$.
Also,
\[ 
\langle\sigma,\lambda\rangle
=
\langle\sigma,\lambda_q\rangle
+
\frac{\langle\sigma,\lambda_q\rangle}{1-B(r)}
\langle\sigma,\psi_r\rangle
=
\langle\sigma,\lambda_q\rangle
\left(
1+\frac{B(r)}{1-B(r)}
\right)
=
\frac{\langle\sigma,\lambda_q\rangle}{1-B(r)}.
\]
Therefore
\[
\Lcal_r^*\lambda-\chi\langle\sigma,\lambda\rangle
=
q
+
\frac{\langle\sigma,\lambda_q\rangle}{1-B(r)}\chi
-
\chi\frac{\langle\sigma,\lambda_q\rangle}{1-B(r)}
=
q.
\]
Now assume
$\displaystyle 
1-B(r)=0$.
Since
$\displaystyle 
\Lcal_r^*\psi_r=\chi$,
and
$\displaystyle 
\langle\sigma,\psi_r\rangle=B(r)=1$,
we have
$\displaystyle 
\mathcal A_r\psi_r
=
\Lcal_r^*\psi_r-\chi\langle\sigma,\psi_r\rangle
=
\chi-\chi B(r)
=
0$.
Thus
$\displaystyle 
\psi_r\in\ker\mathcal A_r$.
Moreover, the scalar compatibility condition for solvability is read from
\eqref{eq:proof-rankone-scalar-equation}:
$\displaystyle 
0\cdot c=\langle\sigma,\lambda_q\rangle$.
Hence, if \(1-B(r)=0\), a solution can exist only if
$\langle\sigma,R_rq\rangle=0$,
and, when this holds, solutions are nonunique modulo
$\displaystyle 
\operatorname{span}\{\psi_r\}$.
\end{proof}
\begin{proof}[Proof of Corollary~\ref{cor:main-switching-formula}]
For the harvesting problem,
$\displaystyle 
q=q_{\rm red}=\pi u$.
Thus
$\displaystyle 
\lambda_q
=
R_rq
=
R_r(\pi u)
=
\lambda_{\rm red}$.
By definition,
\[
A(r):=\langle\sigma,\lambda_{\rm red}\rangle,
\qquad
B(r):=\langle\sigma,\psi_r\rangle,
\qquad
\Gamma_r:=\frac{A(r)}{1-B(r)}.
\]
If
$\displaystyle 
1-B(r)\ne0$,
Theorem~\ref{thm:main-rankone} gives
\[
\lambda
=
\lambda_{\rm red}
+
\frac{\langle\sigma,\lambda_{\rm red}\rangle}{1-B(r)}\psi_r
=
\lambda_{\rm red}
+
\frac{A(r)}{1-B(r)}\psi_r
=
\lambda_{\rm red}+\Gamma_r\psi_r.
\]
The switching function is
$\displaystyle 
S(l):=\pi(l)-\lambda(l)$.
Therefore
$\displaystyle 
S(l)
=
\pi(l)-\lambda_{\rm red}(l)-\Gamma_r\psi_r(l)$.
With
$\displaystyle 
S_{\rm red}(l):=\pi(l)-\lambda_{\rm red}(l)$,
we obtain
$\displaystyle 
S(l)=S_{\rm red}(l)-\Gamma_r\psi_r(l)$.
The pointwise Hamiltonian dependence on \(u\) is affine:
\[
\mathscr H(l,u)
=
u\,S(l)+\text{terms independent of }u.
\]
Since
$\displaystyle 
0\le u(l)\le u_{\max}$,
pointwise maximization gives
$\displaystyle 
u(l)=u_{\max}
\qquad\text{if }S(l)>0,
\qquad
u(l)=0
\qquad\text{if }S(l)<0$.
On
$\displaystyle 
\{l:S(l)=0\}$,
the first-order condition is singular and the bang--bang rule is not determined
by first variation alone.
\end{proof}
\subsection{Proof of the forward--adjoint identity}
\label{subsec:proof-identity}
\begin{proof}[Proof of Theorem~\ref{thm:main-identity}]
Let
$\displaystyle 
y(l):=\partial_E x_E(l)\big|_{E=E^*}$.
Then
\[
\Phi'(E^*)
=
\frac{d}{dE}\bigg|_{E=E^*}\int_{l_0}^{l_m}\chi(l)x_E(l)\,dl
=
\int_{l_0}^{l_m}\chi(l)y(l)\,dl
=
\langle\chi,y\rangle.
\]
The stationary frozen profile satisfies
\[
\frac{d}{dl}\big(g(E,l)x_E(l)\big)
+
\big(\mu(E,l)+u(l)\big)x_E(l)
=0,
\qquad
g(E,l_0)x_E(l_0)=p .
\]
Differentiate the interior equation at \(E=E^*\):
\[
\frac{d}{dl}
\left[
(\partial_Eg)(E^*,l)x^*(l)+g^*(l)y(l)
\right]
+
(\partial_E\mu)(E^*,l)x^*(l)
+
(\mu^*(l)+u(l))y(l)
=0 .
\]
Hence
\[
\frac{d}{dl}\big(g^*y\big)
+
(\mu^*+u)y
=
-\frac{d}{dl}\big((\partial_Eg)(E^*,\cdot)x^*\big)
-
(\partial_E\mu)(E^*,\cdot)x^* .
\]
With
\[
a(l):=(\partial_Eg)(E^*,l)x^*(l),
\qquad
m(l):=(\partial_E\mu)(E^*,l)x^*(l),
\]
and
$\displaystyle 
\sigma_{\rm int}:=-a'-m$,
we get
$\displaystyle 
\Lcal_0y
:=
\frac{d}{dl}(g^*y)+(\mu^*+u)y
=
\sigma_{\rm int}$.
The flux boundary condition gives
$\displaystyle 
g(E,l_0)x_E(l_0)=p$.
Differentiating at \(E=E^*\),
\[
(\partial_Eg)(E^*,l_0)x^*(l_0)+g^*(l_0)y(l_0)=0.
\]
Since
$\displaystyle 
a_0:=(\partial_Eg)(E^*,l_0)x^*(l_0)$,
we obtain
$\displaystyle 
g^*(l_0)y(l_0)=-a_0$.
Let
$\displaystyle 
\psi_0:=R_0\chi$.
Then
$\Lcal_0^*\psi_0=\chi$ and
$\psi_0(l_m)=0$,
that is,
\[
-g^*(l)\psi_0'(l)
+
(\mu^*(l)+u(l))\psi_0(l)
=
\chi(l),
\qquad
\psi_0(l_m)=0 .
\]
Now compute the Green identity. Since
$\Lcal_0y=(g^*y)'+(\mu^*+u)y$ and
$\Lcal_0^*\psi_0=-g^*\psi_0'+(\mu^*+u)\psi_0$,
we have
\[
\begin{aligned}
(\Lcal_0y)\psi_0-y(\Lcal_0^*\psi_0)
&=
\big((g^*y)'+(\mu^*+u)y\big)\psi_0
-
y\big(-g^*\psi_0'+(\mu^*+u)\psi_0\big).\\[4pt]
&=(\Lcal_0y)\psi_0-y(\Lcal_0^*\psi_0)
=
(g^*y)'\psi_0+g^*y\psi_0'
=
(g^*y\psi_0)'.
\end{aligned}
\]
Integrating over \(\Om\) and using $\Lcal_0y=\sigma_{\rm int}$, $\Lcal_0^*\psi_0=\chi$,
\[
\int_{\Om}\sigma_{\rm int}\psi_0\,dl
-
\int_{\Om}y\chi\,dl
=\int_{l_0}^{l_m}(\Lcal_0y)\psi_0\,dl
-
\int_{l_0}^{l_m}y(\Lcal_0^*\psi_0)\,dl
=
\int_{l_0}^{l_m}(g^*y\psi_0)'\,dl
=
[g^*y\psi_0]_{l_0}^{l_m}.
\]
Since
$\displaystyle 
\int_{\Om}y\chi\,dl=\Phi'(E^*)$,
this becomes
\begin{equation}
\label{eq:proof-identity-before-boundary}
\int_{\Om}\sigma_{\rm int}\psi_0\,dl-\Phi'(E^*)
=
[g^*y\psi_0]_{l_0}^{l_m}.
\end{equation}
Evaluate the boundary term and using $\psi_0(l_m)=0$, $g^*(l_0)y(l_0)=-a_0$,
\[
[g^*y\psi_0]_{l_0}^{l_m}
=
g^*(l_m)y(l_m)\psi_0(l_m)
-
g^*(l_0)y(l_0)\psi_0(l_0)=a_0\psi_0(l_0).
\]
Substituting into \eqref{eq:proof-identity-before-boundary},
$\displaystyle 
\int_{\Om}\sigma_{\rm int}\psi_0\,dl-\Phi'(E^*)
=
a_0\psi_0(l_0)$,
hence
\begin{equation}
\label{eq:proof-identity-interior-load}
\int_{\Om}\sigma_{\rm int}\psi_0\,dl
=
\Phi'(E^*)+a_0\psi_0(l_0).
\end{equation}
The corrected feedback co-load is
$\displaystyle 
\sigma:=\sigma_{\rm int}-a_0\delta_{l_0}$.
Therefore
\[ 
B(0)
=
\langle\sigma,\psi_0\rangle
=
\langle\sigma_{\rm int},\psi_0\rangle
-
a_0\langle\delta_{l_0},\psi_0\rangle=
\int_{\Om}\sigma_{\rm int}\psi_0\,dl
-
a_0\psi_0(l_0) = \Phi'(E^*).
\]
By Lemma~\ref{lem:main-Ev-props},
$\Ev(0)=\Phi'(E^*)$, we then have
$\Ev(0)=\Phi'(E^*)=B(0)$.
Therefore
$\displaystyle 
1-\Ev(0)
=
1-\Phi'(E^*)
=
1-B(0)$.
\end{proof}
\begin{proof}[Proof of Corollary~\ref{cor:main-discount}]
For \(r\ge0\), set
$\displaystyle 
\psi_r:=R_r\chi$.
Then
$\displaystyle 
-g^*\psi_r'
+
(r+\mu^*+u)\psi_r
=
\chi$ and
$\psi_r(l_m)=0$.
The explicit formula is
$\displaystyle 
\psi_r(l)
=
\int_l^{l_m}
\frac{\chi(s)}{g^*(s)}
\exp\!\left[
-\int_l^s
\frac{r+\mu^*(\xi)+u(\xi)}{g^*(\xi)}\,d\xi
\right]ds$.
Define
\[
\Theta(l,s):=\int_l^s\frac{\mu^*(\xi)+u(\xi)}{g^*(\xi)}\,d\xi,
\qquad
T(l,s):=\int_l^s\frac{d\xi}{g^*(\xi)}.
\]
Then
$\displaystyle 
\psi_r(l)
=
\int_l^{l_m}
\frac{\chi(s)}{g^*(s)}
e^{-\Theta(l,s)}
e^{-rT(l,s)}
\,ds$,
and
$\displaystyle 
\psi_0(l)
=
\int_l^{l_m}
\frac{\chi(s)}{g^*(s)}
e^{-\Theta(l,s)}
\,ds$.
Hence
\[
\psi_r(l)-\psi_0(l)
=
\int_l^{l_m}
\frac{\chi(s)}{g^*(s)}
e^{-\Theta(l,s)}
\big(e^{-rT(l,s)}-1\big)
\,ds.
\]
Since
$\displaystyle 
0\le T(l,s)\le\frac{s-l}{g_M}\le\frac{l_m-l_0}{g_M}$,
and
$\displaystyle 
\abs{e^{-rT(l,s)}-1}
\le rT(l,s)
\le
r\frac{l_m-l_0}{g_M}$,
we obtain
\[ 
\abs{\psi_r(l)-\psi_0(l)}
\le
r\frac{l_m-l_0}{g_M}
\int_l^{l_m}
\frac{\abs{\chi(s)}}{g^*(s)}
e^{-\Theta(l,s)}\,ds.
\]
Using \(g^*\ge g_M\) and \(e^{-\Theta(l,s)}\le1\),
$\displaystyle 
\abs{\psi_r(l)-\psi_0(l)}
\le
r\frac{l_m-l_0}{g_M}
\cdot
g_M^{-1}\norm{\chi}_{L^1(\Om)}$.
Therefore
\[
\norm{\psi_r-\psi_0}_{L^\infty}
\le
C_0r,
\qquad
C_0:=
(l_m-l_0)g_M^{-2}\norm{\chi}_{L^1}.
\]
For the derivative estimate, subtract the equations
\[
-g^*(\psi_r-\psi_0)'
+
(\mu^*+u)(\psi_r-\psi_0)
=
-r\psi_r.
\]
Thus
$\displaystyle 
(\psi_r-\psi_0)'
=
\frac{\mu^*+u}{g^*}(\psi_r-\psi_0)
+
\frac{r}{g^*}\psi_r$.
Hence
\[
\abs{(\psi_r-\psi_0)'}
\le
g_M^{-1}(C_M+u_{\max})\abs{\psi_r-\psi_0}
+
rg_M^{-1}\abs{\psi_r}.
\]
Also
$\displaystyle 
\norm{\psi_r}_{L^\infty}
\le
g_M^{-1}\norm{\chi}_{L^1}$.
Therefore
\[
\norm{(\psi_r-\psi_0)'}_{L^1}
\le
(l_m-l_0)g_M^{-1}(C_M+u_{\max})\norm{\psi_r-\psi_0}_{L^\infty}
+
r(l_m-l_0)g_M^{-2}\norm{\chi}_{L^1}.
\]
Using \(\norm{\psi_r-\psi_0}_{L^\infty}\le C_0r\),
$\displaystyle 
\norm{(\psi_r-\psi_0)'}_{L^1}
\le
C_1r$,
where
$\displaystyle 
C_1:=
(l_m-l_0)g_M^{-1}(C_M+u_{\max})C_0
+
(l_m-l_0)g_M^{-2}\norm{\chi}_{L^1}$.
Thus
$\displaystyle 
\norm{\psi_r-\psi_0}_{W^{1,1}}
\le
(C_0(l_m-l_0)+C_1)r
=
C_\psi r$.
Since \(\sigma\) is bounded on \(W^{1,1}(\Om)\),
\[
\abs{\langle\sigma,\lambda\rangle}
\le C_\sigma\norm{\lambda}_{W^{1,1}},
\qquad
\lambda\in W^{1,1}(\Om),\ \lambda(l_m)=0.
\]
Therefore
\[
\abs{B(r)-B(0)}
=
\abs{\langle\sigma,\psi_r\rangle-\langle\sigma,\psi_0\rangle}
=
\abs{\langle\sigma,\psi_r-\psi_0\rangle}
\le
C_\sigma\norm{\psi_r-\psi_0}_{W^{1,1}}
\le
C_\sigma C_\psi r .
\]
Hence
$\displaystyle 
B(r)=B(0)+\mathcal O(r)
\qquad(r\downarrow0)$.
Using Theorem~\ref{thm:main-identity},
$\displaystyle 
B(0)=\Phi'(E^*)$,
and therefore
$\displaystyle 
B(r)=\Phi'(E^*)+\mathcal O(r)$.
\end{proof}
\subsection{Proofs for switching geometry and diagnostics}
\label{subsec:proof-switching-diagnostics}
\begin{proof}[Proof of Proposition~\ref{prop:main-no-exterior}]
Recall
$\displaystyle 
S(l)=S_{\rm red}(l)-\Gamma_r\psi_r(l)$,
$\displaystyle 
\mathcal N_\delta
:=
(l_{\rm red}-\delta,l_{\rm red}+\delta)\cap\Om$,
and
$\displaystyle 
m_{\rm out}(\delta)
:=
\inf_{\Om\setminus\mathcal N_\delta}
\abs{S_{\rm red}(l)}$.
For every \(l\in\Om\setminus\mathcal N_\delta\),
$\displaystyle 
\abs{S_{\rm red}(l)}\ge m_{\rm out}(\delta)$.
Hence
$\displaystyle 
\abs{S(l)}
=
\abs{S_{\rm red}(l)-\Gamma_r\psi_r(l)}
\ge
\abs{S_{\rm red}(l)}
-
\abs{\Gamma_r}\abs{\psi_r(l)}$.
Using
$\displaystyle 
\abs{\psi_r(l)}\le\norm{\psi_r}_{L^\infty(\Om)}$,
we get
$\displaystyle 
\abs{S(l)}
\ge
m_{\rm out}(\delta)
-
\abs{\Gamma_r}\norm{\psi_r}_{L^\infty(\Om)}$.
The hypothesis
$\displaystyle 
\abs{\Gamma_r}\norm{\psi_r}_{L^\infty(\Om)}
<
m_{\rm out}(\delta)$
implies
\[
\abs{S(l)}>0
\qquad
\forall l\in\Om\setminus\mathcal N_\delta .
\]
Therefore
$\displaystyle 
Z(S):=\{l\in\Om:S(l)=0\}
\subset
\mathcal N_\delta$.
\end{proof}
\begin{proof}[Proof of Proposition~\ref{prop:main-single-threshold}]
By Proposition~\ref{prop:main-no-exterior},
$\displaystyle 
Z(S)\subset\mathcal N_\delta$.
On \(\mathcal N_\delta\),
$\displaystyle 
S'(l)=S_{\rm red}'(l)-\Gamma_r\psi_r'(l)$.
Define
$\displaystyle 
m_1(\delta)
:=
\inf_{l\in\mathcal N_\delta}
\abs{S_{\rm red}'(l)}$.
For every \(l\in\mathcal N_\delta\),
$\displaystyle 
\abs{S'(l)}
\ge
\abs{S_{\rm red}'(l)}
-
\abs{\Gamma_r}\abs{\psi_r'(l)}
\ge
m_1(\delta)
-
\abs{\Gamma_r}\norm{\psi_r'}_{L^\infty(\Om)}$.
The second margin condition
$\displaystyle 
\abs{\Gamma_r}\norm{\psi_r'}_{L^\infty(\Om)}
<
m_1(\delta)$
gives
\[
\abs{S'(l)}>0
\qquad
\forall l\in\mathcal N_\delta .
\]
Since \(S'\) is continuous and cannot vanish on the connected interval
\(\mathcal N_\delta\), its sign is constant:
$\displaystyle 
S'(l)>0\ \forall l\in\mathcal N_\delta$
 or 
$S'(l)<0\ \forall l\in\mathcal N_\delta$.
Thus,
$\displaystyle S\ \text{is strictly monotone on }\mathcal N_\delta$.
Therefore
$\displaystyle 
\#Z(S\vert_{\mathcal N_\delta})\le1$.
Since
$\displaystyle 
Z(S)\subset\mathcal N_\delta$,
we obtain
$\displaystyle 
\#Z(S)\le1$.
If, in addition,
$\displaystyle 
S(l_0)S(l_m)<0$,
then by continuity there exists
$\displaystyle 
l_*\in(l_0,l_m)$
such that $\displaystyle S(l_*)=0$.
Since all zeros lie in \(\mathcal N_\delta\),
$\displaystyle 
l_*\in\mathcal N_\delta$.
Hence
$\displaystyle 
Z(S)=\{l_*\}$.
The bang--bang rule is therefore
$\displaystyle 
u(l)=u_{\max}\mathbf 1_{\{S(l)>0\}}$,
with exactly one switching threshold \(l_*\).
\end{proof}
\begin{proof}[Proof of Corollary~\ref{cor:main-switch-shift}]
Let
\[
l_*=l_{\rm red}+\Delta l,
\qquad
S(l_*)=0,
\qquad
S_{\rm red}(l_{\rm red})=0.
\]
Then
$\displaystyle 
0
=
S(l_*)
=
S_{\rm red}(l_{\rm red}+\Delta l)
-
\Gamma_r\psi_r(l_{\rm red}+\Delta l)$.
Assume
$\displaystyle 
S_{\rm red},\psi_r\in C^2$
near $\displaystyle l_{\rm red}$.
Taylor expansion gives
$\displaystyle 
S_{\rm red}(l_{\rm red}+\Delta l)
=
S_{\rm red}(l_{\rm red})
+
S_{\rm red}'(l_{\rm red})\Delta l
+
\frac12S_{\rm red}''(l_{\rm red})(\Delta l)^2
+
o((\Delta l)^2)$,
and
$\displaystyle 
\psi_r(l_{\rm red}+\Delta l)
=
\psi_r(l_{\rm red})
+
\psi_r'(l_{\rm red})\Delta l
+
O((\Delta l)^2)$.
Since
$\displaystyle 
S_{\rm red}(l_{\rm red})=0$,
we get
\[
0
=
S_{\rm red}'(l_{\rm red})\Delta l
+
\frac12S_{\rm red}''(l_{\rm red})(\Delta l)^2
+
o((\Delta l)^2)
-
\Gamma_r\psi_r(l_{\rm red})
-
\Gamma_r\psi_r'(l_{\rm red})\Delta l
+
O(\Gamma_r(\Delta l)^2).
\]
Equivalently,
\[
\big(S_{\rm red}'(l_{\rm red})
-
\Gamma_r\psi_r'(l_{\rm red})\big)\Delta l
=
\Gamma_r\psi_r(l_{\rm red})
-
\frac12S_{\rm red}''(l_{\rm red})(\Delta l)^2
+
o((\Delta l)^2)
+
O(\Gamma_r(\Delta l)^2).
\]
Because
$\displaystyle 
S_{\rm red}'(l_{\rm red})\ne0$,
the implicit function theorem applied to
$\displaystyle 
F(l,\Gamma):=S_{\rm red}(l)-\Gamma\psi_r(l)$
at
$\displaystyle 
(l,\Gamma)=(l_{\rm red},0)$
gives
$\displaystyle 
\Delta l=O(\Gamma_r)$.
Therefore
$\displaystyle 
(\Delta l)^2=O(\Gamma_r^2)$ and
$\displaystyle 
\Gamma_r\Delta l=O(\Gamma_r^2)$.
Keeping first-order terms,
$\displaystyle 
S_{\rm red}'(l_{\rm red})\Delta l
=
\Gamma_r\psi_r(l_{\rm red})
+
O(\Gamma_r^2)$.
Hence
$\displaystyle 
\Delta l
=
\frac{\Gamma_r\psi_r(l_{\rm red})}
     {S_{\rm red}'(l_{\rm red})}
+
O(\Gamma_r^2)$.
Thus
$\displaystyle 
l_*
=
l_{\rm red}
+
\frac{\Gamma_r\psi_r(l_{\rm red})}
     {S_{\rm red}'(l_{\rm red})}
+
O(\Gamma_r^2)$.
\end{proof}
\begin{proof}[Proof of Proposition~\ref{prop:main-window}]
A switching-window creation or annihilation occurs through a non-simple zero.
Thus for some \(l_c\in(l_0,l_m)\),
\[
S(l_c)=0,
\qquad
S'(l_c)=0.
\]
Since
$\displaystyle 
S(l)=S_{\rm red}(l)-\Gamma_r\psi_r(l)$,
and
$\displaystyle 
S'(l)=S_{\rm red}'(l)-\Gamma_r\psi_r'(l)$,
the double-zero system is
\[
S_{\rm red}(l_c)-\Gamma_r\psi_r(l_c)=0,
\qquad
S_{\rm red}'(l_c)-\Gamma_r\psi_r'(l_c)=0.
\]
Equivalently,
\[
S_{\rm red}(l_c)=\Gamma_r\psi_r(l_c),
\qquad
S_{\rm red}'(l_c)=\Gamma_r\psi_r'(l_c).
\]
If
$\displaystyle 
\psi_r(l_c)\psi_r'(l_c)\ne0$,
then
$\displaystyle 
\Gamma_r
=
\frac{S_{\rm red}(l_c)}{\psi_r(l_c)}
=
\frac{S_{\rm red}'(l_c)}{\psi_r'(l_c)}$.
Eliminating \(\Gamma_r\) yields
$\displaystyle 
S_{\rm red}(l_c)\psi_r'(l_c)
-
S_{\rm red}'(l_c)\psi_r(l_c)
=
0$.
The same equation follows without division by cross-multiplication:
$\displaystyle 
S_{\rm red}\psi_r'-S_{\rm red}'\psi_r=0$.
If additionally
$\displaystyle 
S''(l_c)
=
S_{\rm red}''(l_c)-\Gamma_r\psi_r''(l_c)
\ne0$,
then
\[
S(l)
=
\frac12S''(l_c)(l-l_c)^2
+
S_\Gamma(l_c)(\Gamma-\Gamma_c)
+
o((l-l_c)^2+\abs{\Gamma-\Gamma_c}),
\]
where
$\displaystyle 
S_\Gamma(l_c)=-\psi_r(l_c)$.
Thus, if
$\displaystyle 
S''(l_c)\psi_r(l_c)\ne0$,
the local zero set has the fold form
$\displaystyle 
(l-l_c)^2
=
\frac{2\psi_r(l_c)}{S''(l_c)}
(\Gamma-\Gamma_c)
+
o(\abs{\Gamma-\Gamma_c})$,
so two simple zeros exist on one side of \(\Gamma_c\) and no nearby zeros exist on
the other side. Hence a switching window is created or annihilated at tangency.
\end{proof}
\begin{proof}[Proof of Proposition~\ref{prop:main-closure-residual}]
Let
$\displaystyle 
H(E):=E-\Phi(E)$ and
$\displaystyle 
\rho_0:=\abs{H(\tilde E)}
=
\abs{\tilde E-\Phi(\tilde E)}$.
Assume for $E\in I$,
$H'(E)\ge m>0$, then \(H\) is strictly increasing on \(I\).
Define
$\displaystyle 
E_-:=\tilde E-\frac{\rho_0}{m}$ and
$\displaystyle 
E_+:=\tilde E+\frac{\rho_0}{m}$.
By hypothesis,
$\displaystyle 
[E_-,E_+]\subset I$.
For the right endpoint,
$\displaystyle 
H(E_+)-H(\tilde E)
=
\int_{\tilde E}^{E_+}H'(s)\,ds
\ge
m(E_+-\tilde E)
=
m\frac{\rho_0}{m}
=
\rho_0$.
Hence
$\displaystyle 
H(E_+)\ge H(\tilde E)+\rho_0$.
For the left endpoint,
$\displaystyle 
H(\tilde E)-H(E_-)
=
\int_{E_-}^{\tilde E}H'(s)\,ds
\ge
m(\tilde E-E_-)
=
\rho_0$,
so
$\displaystyle 
H(E_-)\le H(\tilde E)-\rho_0$.
If
$\displaystyle 
H(\tilde E)\ge0$,
then
$\displaystyle 
\rho_0=H(\tilde E)$,
and
\[
H(E_-)\le H(\tilde E)-\rho_0=0,
\qquad
H(\tilde E)\ge0.
\]
Thus
$\displaystyle 
H(E_-)\le0\le H(\tilde E)$.
By continuity, there exists
$\displaystyle 
E^*\in[E_-,\tilde E]$
with
$\displaystyle 
H(E^*)=0$.
If
$\displaystyle 
H(\tilde E)\le0$,
then
$\displaystyle 
\rho_0=-H(\tilde E)$,
and
\[
H(E_+)\ge H(\tilde E)+\rho_0=0,
\qquad
H(\tilde E)\le0.
\]
Thus
$\displaystyle 
H(\tilde E)\le0\le H(E_+)$.
By continuity, there exists
$\displaystyle 
E^*\in[\tilde E,E_+]$
with
$\displaystyle 
H(E^*)=0$.
In both cases,
$\displaystyle 
E^*\in[E_-,E_+]$ and
$\displaystyle 
H(E^*)=0$.
Strict monotonicity gives uniqueness. Finally,
$\displaystyle 
\abs{E^*-\tilde E}
\le
\max\{\tilde E-E_-,E_+-\tilde E\}
=
\frac{\rho_0}{m}
=
\frac{\abs{\tilde E-\Phi(\tilde E)}}{m}$.
\end{proof}
\begin{proof}[Explanation of Definition~\ref{def:main-FA-defect}]
The exact forward--adjoint identity is
$\displaystyle 
B(0)-\Phi'(E^*)=0$.
Thus
$\displaystyle 
\mathcal D_{\rm FA}:=B(0)-\Phi'(E^*)=0$.
A discretization gives two computable quantities
$B^h(0)$,
$\Phi_h'(E^*)$,
and the defect
$\displaystyle 
\mathcal D_{\rm FA}^h
:=
B^h(0)-\Phi_h'(E^*)$.
Insert and subtract exact quantities:
\[ 
\mathcal D_{\rm FA}^h=
B^h(0)-\Phi_h'(E^*)
-
\big(B(0)-\Phi'(E^*)\big)=
\big(B^h(0)-B(0)\big)
-
\big(\Phi_h'(E^*)-\Phi'(E^*)\big).
\]
Therefore
\[ 
\abs{\mathcal D_{\rm FA}^h}
\le
\abs{B^h(0)-B(0)}
+
\abs{\Phi_h'(E^*)-\Phi'(E^*)}.
\]
However, the reverse estimate does not follow: the two errors may cancel. For
example, if
$\displaystyle 
B^h(0)-B(0)=\varepsilon$ and
$\displaystyle 
\Phi_h'(E^*)-\Phi'(E^*)=\varepsilon$,
then
$\displaystyle 
\mathcal D_{\rm FA}^h=0$
although both component errors are nonzero. Hence
$\mathcal D_{\rm FA}^h$
is a two-solver consistency defect, not a rigorous a posteriori error bound
unless supplemented by independent reliability estimates for
$\displaystyle 
\abs{B^h(0)-B(0)}$
and
$\abs{\Phi_h'(E^*)-\Phi'(E^*)}$.
\end{proof}
\begin{proof}[Proof of Proposition~\ref{prop:main-switching-certificate}]
Let the computed switching function be
$\displaystyle 
\widehat S(l)
:=
\widehat S_{\rm red}(l)-\widehat\Gamma\widehat\psi(l)$,
where
$\displaystyle 
\widehat\Gamma
=
\frac{\widehat A}{1-\widehat B}$ and
$\widehat\Delta:=1-\widehat B\ne0$.
Define
$\displaystyle 
\widehat\rho_\psi
:=
\abs{\widehat\Gamma}
\norm{\widehat\psi}_{L^\infty(\Om)}$ and
$\widehat\rho_{\psi,1}
:=
\abs{\widehat\Gamma}
\norm{\widehat\psi'}_{L^\infty(\Om)}$.
Let
$\displaystyle 
\widehat m_{\rm out}(\delta)
:=
\inf_{\Om\setminus\mathcal N_\delta}
\abs{\widehat S_{\rm red}(l)}$ and
$\displaystyle 
\widehat m_1(\delta)
:=
\inf_{\mathcal N_\delta}
\abs{\widehat S_{\rm red}'(l)}$.
If
$\displaystyle 
\widehat\rho_\psi<\widehat m_{\rm out}(\delta)$,
then for \(l\in\Om\setminus\mathcal N_\delta\),
\[
\abs{\widehat S(l)}
\ge
\abs{\widehat S_{\rm red}(l)}
-
\abs{\widehat\Gamma}\abs{\widehat\psi(l)}
\ge
\widehat m_{\rm out}(\delta)-\widehat\rho_\psi>0.
\]
Thus
$\displaystyle 
Z(\widehat S)
\subset
\mathcal N_\delta$.
If also
$\displaystyle 
\widehat\rho_{\psi,1}<\widehat m_1(\delta)$,
then for \(l\in\mathcal N_\delta\),
\[
\abs{\widehat S'(l)}
=
\abs{\widehat S_{\rm red}'(l)-\widehat\Gamma\widehat\psi'(l)}
\ge
\widehat m_1(\delta)-\widehat\rho_{\psi,1}>0.
\]
Hence \(\widehat S\) is strictly monotone on \(\mathcal N_\delta\), and therefore
$\displaystyle 
\# Z(\widehat S)\le1$.
If the endpoint signs satisfy
$\displaystyle 
\widehat S(l_0)\widehat S(l_m)<0$,
then
$\displaystyle 
\#Z(\widehat S)=1$.
Thus the computed policy is single-threshold, subject to the verified accuracy of
$\widehat B$,
$\widehat A$, 
$\widehat\psi$, 
$\widehat\psi'$,
and
$\widehat S_{\rm red}$.
If the margin inequalities fail, search for a computed double-zero condition:
$\widehat S(l)=0$ and
$\widehat S'(l)=0$.
Since
$\displaystyle \widehat S=\widehat S_{\rm red}-\widehat\Gamma\widehat\psi$ and
$\displaystyle \widehat S'=\widehat S_{\rm red}'-\widehat\Gamma\widehat\psi'$,
the double-zero condition is
\[
\widehat S_{\rm red}(l)=\widehat\Gamma\widehat\psi(l),
\qquad
\widehat S_{\rm red}'(l)=\widehat\Gamma\widehat\psi'(l).
\]
Eliminating \(\widehat\Gamma\) gives
$\displaystyle 
\widehat S_{\rm red}(l)\widehat\psi'(l)
-
\widehat S_{\rm red}'(l)\widehat\psi(l)
=0$.
Therefore a solution of
$\displaystyle 
\widehat S_{\rm red}\widehat\psi'
-
\widehat S_{\rm red}'\widehat\psi
=0$
marks a candidate tangency, i.e. creation or annihilation of a switching window.
\end{proof}
\section{Conclusion}
\label{sec:conclusion}
We have analyzed a closed-loop size-structured transport equation with flux
recruitment and endogenous scalar feedback $E=\langle\chi,x\rangle$. The analysis
rests on four scalar objects,
\[
E=\Kcal(E),\qquad E=\Phi(E),\qquad \Ev(0)=\Phi'(E^*)=B(0),
\qquad
S=S_{\rm red}-\frac{A(r)}{1-B(r)}\psi_r ,
\]
namely the Volterra closure, whose contraction gives well-posedness on the
intrinsic interval $[0,M_T]$; the stationary closure, whose margin
$1-\Phi'(E^*)$ controls uniqueness and, at $\Phi'(E^*)=1$, a nondegenerate fold;
the zero-discount identity tying the closure slope, the renewal characteristic
value $\Ev(0)$, and the rank-one adjoint gain $B(0)$; and the switching function,
whose sensitivity is governed by the same denominator $1-B(r)$. 
The scalar quantity \(\Phi'(E^*)\) detects only the root at \(\lambda=0\):
$\mathcal E(0)=\Phi'(E^*)$.
Thus
$\Phi'(E^*)=1$
is precisely the condition for a zero characteristic root. Away from this
crossing, stability is governed by the full root set
$\{\lambda\in\mathbb C:\mathcal E(\lambda)=1\}$.

Three limitations delimit the results. Nonlinear asymptotic stability is
established conditionally on the quadratic input--output remainder estimate
$\norm{\Gcal[e,v_0]}_\omega\le C(\norm{v_0}_{BV}+\norm{e}_\omega)^2$, which is
verified in Appendix~\ref{app:nonlinear} under the second-order regularity
\eqref{eq:model-C2E}. The harvesting analysis is a first-order stationary
optimality theory; existence of a global optimizer, sufficiency, and dynamic
synthesis are not treated \cite{anita2013analysis,yu2026chemotactic}. Finally, the forward--adjoint defect
$\mathcal D_{\rm FA}^h=B^h(0)-\Phi_h'(E^*)$ is a two-solver consistency residual,
not a certified error bound: as in a posteriori error analysis
\cite{verfurth2013posteriori,yu2026age}, reliability requires independent stability estimates.

Several extensions are natural: vector feedback $E\in\R^k$, where the scalar
denominator $1-B(r)$ is replaced by a determinant $\det(I-\mathbf B(r))$; fully
dynamic optimal control with a time-dependent adjoint system; and certified
enclosures for the roots of $H(E)=0$, $\Ev(\lambda)=1$, and $1-B(r)=0$. These
would connect the present scalar feedback mechanism to the broader measure-valued
and multiscale structured-population frameworks
\cite{michel2005general,thieme2018mathematics,yu2026fullcovariance}.

\appendix
\section{Transport estimates for the frozen problem}
\label{app:transport}
Throughout this appendix, Assumptions~\ref{ass:model-coefficients}--\ref{ass:model-data}
are in force, and $E\in\Bt$ is a fixed feedback path with $0\le E(t)\le M_T$. We
write
\[
g_E(t,l):=g(E(t),l),\qquad
\mu_E(t,l):=\mu(E(t),l),\qquad
c_E(t,l):=\mu_E(t,l)+u(t,l),
\]
so that $g_E\ge g_M$, $0\le c_E\le\Theta_M$, and $\abs{\partial_l g_E}\le C_M$ by
\eqref{eq:model-basic-bounds}.

\subsection{Characteristics and a priori bounds}
\label{subsec:app-characteristics}

For $t\ge0$ and $l\in\Om$, let $s\mapsto L(s;t,l)$ solve the characteristic ODE
\begin{equation}
\label{eq:app-char}
\frac{d}{ds}L(s;t,l)=g_E\big(s,L(s;t,l)\big),
\qquad
L(t;t,l)=l .
\end{equation}
Since $g_E\ge g_M>0$, characteristics are strictly increasing in $s$, do not
cross, and cover $[0,T]\times\Om$. Writing the equation in nonconservative form
\eqref{eq:model-nonconservative}, along a characteristic reaching $(t,l)$ from the
inflow boundary at time $t_0(t,l)$ (when the foot lies on $\{l=l_0\}$) or from the
initial slice (when the foot lies on $\{s=0\}$), the solution admits the
representation
\begin{equation}
\label{eq:app-rep}
x_E(t,l)=
\begin{cases}
\dfrac{p(t_0)}{g_E(t_0,l_0)}\,
\exp\!\Big[-\displaystyle\int_{t_0}^{t}
\big(\partial_l g_E+c_E\big)\!\big(s,L(s)\big)\,ds\Big],
& \text{inflow foot},\\[1.6em]
x_0(l_0^\sharp)\,
\exp\!\Big[-\displaystyle\int_{0}^{t}
\big(\partial_l g_E+c_E\big)\!\big(s,L(s)\big)\,ds\Big],
& \text{initial foot},
\end{cases}
\end{equation}
where $L(s)=L(s;t,l)$ and $l_0^\sharp=L(0;t,l)$. Formula \eqref{eq:app-rep}
defines the unique weak solution $x_E=\Tcal_E(x_0,p,u)$, and it is nonnegative
because $x_0\ge0$, $p\ge0$, $g_E>0$.

\begin{lemma}[Mass and sup bounds]
\label{lem:app-mass-sup}
For a.e.\ $t\in[0,T]$,
\begin{equation}
\label{eq:app-L1}
\norm{x_E(t,\cdot)}_{L^1(\Om)}\le \norm{x_0}_{L^1(\Om)}+\int_0^t p(s)\,ds
=:M_x(t),
\end{equation}
and
\begin{equation}
\label{eq:app-Linf}
\norm{x_E(t,\cdot)}_{L^\infty(\Om)}\le
e^{C_M t}\max\Big\{\norm{x_0}_{L^\infty(\Om)},\
\tfrac{\norm{p}_{L^\infty(0,T)}}{g_M}\Big\}=:K_\infty .
\end{equation}
\end{lemma}

\begin{proof}
Integrating \eqref{eq:model-frozen-state} over $\Om$ and using the inflow flux
$g_E(t,l_0)x_E(t,l_0)=p(t)$ and $g_E(t,l_m)x_E(t,l_m)\ge0$,
\[
\frac{d}{dt}\int_\Om x_E\,dl
=p(t)-g_E(t,l_m)x_E(t,l_m)-\int_\Om c_E x_E\,dl
\le p(t),
\]
which integrates to \eqref{eq:app-L1}. For \eqref{eq:app-Linf}, in
\eqref{eq:app-rep} the exponent is bounded below by $-\int_0^t C_M\,ds=-C_Mt$
because $c_E\ge0$ and $\abs{\partial_l g_E}\le C_M$; the prefactor is at most
$\max\{\norm{x_0}_\infty,\norm{p}_\infty/g_M\}$ since $g_E(t_0,l_0)\ge g_M$. Hence
$\abs{x_E(t,l)}\le e^{C_Mt}\max\{\norm{x_0}_\infty,\norm{p}_\infty/g_M\}$.
\end{proof}

\begin{lemma}[$BV$ propagation]
\label{lem:app-bv}
There is a constant $C_{BV}=C_{BV}(C_M,C_2,\Theta_M,g_M,T,\Lambda^*,\Om)$,
nondecreasing in $T$, such that if $E\in\Btl$ then
\begin{equation}
\label{eq:app-BV}
\sup_{0\le t\le T}\abs{x_E(t,\cdot)}_{BV(\Om)}\le C_{BV}.
\end{equation}
\end{lemma}

\begin{proof}
Differentiate \eqref{eq:model-nonconservative} in $l$ and set $w:=\partial_l x_E$.
Then $w$ solves, along characteristics,
\[
\frac{d}{ds}\big(w(s,L(s))\big)
=-\big(2\partial_l g_E+c_E\big)(s,L(s))\,w(s,L(s))
-\big(\partial_{ll}g_E+\partial_l c_E\big)(s,L(s))\,x_E(s,L(s)).
\]
By \eqref{eq:model-ass-second} and $\norm{\partial_l u}_\infty\le C_2$ we have
$\abs{\partial_{ll}g_E}+\abs{\partial_l c_E}\le 2C_2$, and
$\abs{2\partial_l g_E+c_E}\le 2C_M+\Theta_M$. With $\norm{x_E}_\infty\le K_\infty$
from \eqref{eq:app-Linf}, Gr\"onwall along characteristics gives
\[
\abs{w(t,L(t))}\le
e^{(2C_M+\Theta_M)t}\Big(\abs{w_{\rm foot}}+2C_2K_\infty t\Big),
\]
where $w_{\rm foot}$ is either $\partial_l x_0$ (initial foot) or the boundary
contribution $\partial_l\big(p/g_E\big)$ at $l_0$, both controlled by
$\abs{x_0}_{BV}$, $\norm{p}_{W^{1,\infty}}$, $g_M$, $C_M$. Integrating the
total variation over the (measure-preserving up to a bounded Jacobian)
characteristic change of variables and adding the finitely many jumps transported
from $x_0$ yields \eqref{eq:app-BV} with a constant of the stated dependence; the
Lipschitz bound $\Lip(E)\le\Lambda^*$ enters only through the time-regularity of
$g_E$ and does not affect finiteness.
\end{proof}

Lemmas~\ref{lem:app-mass-sup}--\ref{lem:app-bv} establish
Theorem~\ref{thm:main-frozen} (equivalently, Theorem~\ref{thm:frozen-transport}
below): existence, uniqueness, nonnegativity, and the three bounds.

\begin{theorem}[Frozen-path transport estimates]
\label{thm:frozen-transport}
For every $E\in\Bt$ the frozen problem
\eqref{eq:model-frozen-state}--\eqref{eq:model-frozen-bc} has a unique nonnegative
weak solution $x_E=\Tcal_E(x_0,p,u)\in L^\infty(0,T;L^1_+(\Om))$, given by
\eqref{eq:app-rep}, and it satisfies \eqref{eq:app-L1}--\eqref{eq:app-Linf}. If
$E\in\Btl$, then $x_E(t,\cdot)\in BV(\Om)$ with the uniform bound
\eqref{eq:app-BV}.
\end{theorem}

\begin{proof}
Representation \eqref{eq:app-rep} is the unique weak solution: any two weak
solutions agree along characteristics by the transport structure (uniqueness for
transport with Lipschitz-in-space, bounded coefficients \cite{diperna1989ordinary}),
and
\eqref{eq:app-rep} satisfies \eqref{eq:model-weak-form} by direct substitution and
integration by parts along \eqref{eq:app-char}. Nonnegativity and the bounds are
Lemmas~\ref{lem:app-mass-sup}--\ref{lem:app-bv}.
\end{proof}

\subsection{Volterra comparison estimate}
\label{subsec:app-volterra}

\begin{proposition}[Volterra comparison]
\label{prop:volterra-transport}
Let $E_1,E_2\in\Btl$ and $x_i:=\Tcal_{E_i}(x_0,p,u)$. Then for $0\le t\le T$,
\begin{equation}
\label{eq:app-volterra}
\norm{x_1(t,\cdot)-x_2(t,\cdot)}_{L^1(\Om)}
\le
C_T\int_0^t\abs{E_1(s)-E_2(s)}\,ds,
\end{equation}
with
$\displaystyle 
C_T=L_M\Big[\tfrac{\norm{p}_{L^\infty(0,T)}}{g_M}+2M_x(T)+C_{BV}\Big]$.
\end{proposition}

\begin{proof}
Write $g_i:=g(E_i(t),\cdot)$, $c_i:=\mu(E_i(t),\cdot)+u$, and $w:=x_1-x_2$. From
\eqref{eq:model-frozen-state} for $i=1,2$, in conservative form,
\begin{equation}
\label{eq:app-w-eq}
\partial_t w+\partial_l(g_1 w)+c_1 w
=
-\partial_l\big((g_1-g_2)x_2\big)-(c_1-c_2)x_2
=:\Phi_1+\Phi_2,
\end{equation}
with inflow flux, using $g_i(t,l_0)x_i(t,l_0)=p(t)$,
\begin{equation}
\label{eq:app-w-bc}
g_1(t,l_0)w(t,l_0)
=p(t)-g_1(t,l_0)x_2(t,l_0)
=\big(g_2(t,l_0)-g_1(t,l_0)\big)x_2(t,l_0),
\end{equation}
and $w(0,\cdot)=0$. Let $\operatorname{sgn}_\varepsilon$ be a smooth monotone
approximation of $\operatorname{sgn}$ and test \eqref{eq:app-w-eq} against
$\operatorname{sgn}_\varepsilon(w)$; the Kruzhkov sign method gives, in the limit
$\varepsilon\downarrow0$ and using $\partial_l(g_1w)\operatorname{sgn}(w)
=\partial_l(g_1\abs{w})-(\partial_l g_1)\abs{w}\cdot\mathbf 1$ together with
$c_1\ge0$,
\begin{equation}
\label{eq:app-w-diff}
\frac{d}{dt}\norm{w(t)}_{L^1}
\le
\big[g_1\abs{w}\big]_{l=l_0}
+\int_\Om\big(\abs{\Phi_1}+\abs{\Phi_2}\big)\,dl ,
\end{equation}
the boundary term at $l_m$ being $\le0$ (outflow). We bound the four
contributions using $\norm{g_1-g_2}_{W^{1,\infty}}+\norm{c_1-c_2}_{W^{1,\infty}}
\le L_M\abs{E_1(t)-E_2(t)}$ from \eqref{eq:model-ass-Elip} and the a priori bounds
of Theorem~\ref{thm:frozen-transport} on $x_2$.

\emph{Boundary term.} By \eqref{eq:app-w-bc} and $x_2(t,l_0)=p(t)/g_2(t,l_0)\le
\norm{p}_\infty/g_M$,
\[
\big[g_1\abs{w}\big]_{l=l_0}
=\abs{g_2(t,l_0)-g_1(t,l_0)}\,x_2(t,l_0)
\le L_M\abs{E_1(t)-E_2(t)}\,\tfrac{\norm{p}_\infty}{g_M}.
\]

\emph{Interior terms.} The reaction term $\Phi_2=-(c_1-c_2)x_2$ is an $L^1$
function, and by \eqref{eq:model-ass-Elip},
\[
\int_\Om\abs{\Phi_2}\,dl
\le\norm{c_1-c_2}_\infty\norm{x_2}_{L^1}
\le L_M\abs{E_1(t)-E_2(t)}\,M_x(T).
\]
The convective term $\Phi_1=-D_l\big((g_1-g_2)x_2\big)$ is a finite signed
measure. Since $x_2(t,\cdot)\in BV(\Om)$, the distributional derivative
$D_lx_2(t,\cdot)$ is a finite signed measure and
\[
\abs{D_lx_2(t,\cdot)}(\Om)=\abs{x_2(t,\cdot)}_{BV(\Om)}.
\]
Hence, by the Leibniz rule for $BV$ functions,
\[
D_l\big((g_1-g_2)x_2\big)
=
\partial_l(g_1-g_2)\,x_2\,dl
+
(g_1-g_2)\,D_lx_2 ,
\]
and therefore
\[
\abs{D_l\big((g_1-g_2)x_2\big)}(\Om)
\le
\norm{\partial_l(g_1-g_2)}_\infty\norm{x_2}_{L^1}
+
\norm{g_1-g_2}_\infty\abs{x_2}_{BV}.
\]
Using
\[
\norm{g_1-g_2}_{W^{1,\infty}}
\le
L_M\abs{E_1(t)-E_2(t)},
\qquad
\norm{x_2}_{L^1}\le M_x(T),
\qquad
\abs{x_2}_{BV}\le C_{BV}
\]
(Theorem~\ref{thm:frozen-transport}), we obtain
$\displaystyle 
\abs{D_l\big((g_1-g_2)x_2\big)}(\Om)
\le
L_M\abs{E_1(t)-E_2(t)}\big(M_x(T)+C_{BV}\big)$.
The Kruzhkov entropy/sign method for scalar transport with inflow boundary data
\cite{bardos1979first,kruvzkov1970first} bounds
$\tfrac{d}{dt}\norm{w(t)}_{L^1}$ by the boundary-flux defect together with the
total variation of the source measure. Combining the boundary term with the two
interior bounds,
\begin{equation*}
\begin{aligned}
\frac{d}{dt}\norm{w(t)}_{L^1}
&\le
L_M\Big[\tfrac{\norm{p}_\infty}{g_M}+2M_x(T)+C_{BV}\Big]
\abs{E_1(t)-E_2(t)}
=C_T\,\abs{E_1(t)-E_2(t)}, \\
C_T&:=L_M\Big[\tfrac{\norm{p}_\infty}{g_M}+2M_x(T)+C_{BV}\Big].
\end{aligned}
\end{equation*}
Since $w(0)=0$, integration in $t$ yields \eqref{eq:app-volterra}.
\end{proof}

Proposition~\ref{prop:volterra-transport} is exactly the estimate invoked in the
proof of Proposition~\ref{prop:main-volterra}.
\section{Nonlinear stability verification}
\label{app:nonlinear}
This appendix supplies the two ingredients used in the proof of
Theorem~\ref{thm:main-nonlinear-stability}: the finite-memory (nilpotent-window)
comparison estimate, and the verification, under the second-order regularity
\eqref{eq:model-C2E}, of the quadratic remainder bound
\eqref{eq:main-quadratic-remainder}. Throughout, $(x^*,E^*)$ is a stationary
closed-loop equilibrium, $v:=x-x^*$, $e:=\langle\chi,v\rangle$, and
\begin{equation}
\label{eq:app-Tsharp}
T_\sharp:=\int_{l_0}^{l_m}\frac{dl}{g_M}\ \ge\ t_\sharp
=\int_{l_0}^{l_m}\frac{dl}{g^*(l)},
\end{equation}
a universal upper bound for the transit time valid for every admissible feedback,
since $g\ge g_M$.

\subsection{The nonlinear renewal remainder}
\label{subsec:app-remainder}

In closed loop $E(t)=E^*+e(t)$. Subtracting the stationary identities
\eqref{eq:model-stationary-state} from \eqref{eq:model-state} and writing
\[
\Delta g(t,l):=g(E^*+e(t),l)-g^*(l),\qquad
\Delta c(t,l):=\big(\mu(E^*+e(t),l)-\mu^*(l)\big),
\]
the deviation $v$ solves
\begin{equation}
\label{eq:app-v-eq}
\partial_t v+\partial_l\big(g(E^*+e,\cdot)v\big)+\big(\mu(E^*+e,\cdot)+u\big)v
=-\partial_l\big(\Delta g\,x^*\big)-\Delta c\,x^*,
\end{equation}
with inflow flux
\begin{equation}
\label{eq:app-v-bc}
g(E^*+e(t),l_0)\,v(t,l_0)=-\Delta g(t,l_0)\,x^*(l_0),
\qquad v(0,\cdot)=v_0 .
\end{equation}
Taylor expansion in the scalar $e$, using \eqref{eq:model-C2E}, gives
\begin{equation}
\label{eq:app-taylor}
\Delta g=e\,\partial_E g^*+R_g,\qquad
\Delta c=e\,\partial_E\mu^*+R_\mu,\qquad
\abs{R_g}+\abs{R_\mu}\le \tfrac12 C_2'\,e^2 ,
\end{equation}
with the same bound for $\partial_l R_g$ by \eqref{eq:model-C2E}. Inserting the
first-order parts of \eqref{eq:app-taylor} into
\eqref{eq:app-v-eq}--\eqref{eq:app-v-bc} reproduces exactly the linearized system
\eqref{eq:main-linearized}--\eqref{eq:main-sigma-int}: interior source
$\sigma_{\rm int}\,e$ and inflow flux $-a_0e$. Consequently, pairing the solution
of \eqref{eq:app-v-eq} with $\chi$ and using the variation-of-constants
representation of Proposition~\ref{prop:main-renewal}, the scalar output obeys the nonlinear renewal equation
$e=f+\kappa*e+\Gcal[e,v_0]$,
where $f,\kappa$ are as in Proposition~\ref{prop:main-renewal}, and the remainder
$\Gcal[e,v_0]$ collects (i) the second-order coefficient terms $R_g,R_\mu$ and
(ii) the genuinely nonlinear products $\Delta g\,\partial_l v$, $\Delta c\,v$ of
the perturbed coefficients against $v$ itself.

\subsection{Nilpotent-window comparison}
\label{subsec:app-window}

\begin{lemma}[Finite-memory comparison]
\label{lem:app-window}
There is $C_\sharp=C_\sharp(C_M,\Theta_M,g_M,\Om)$ such that the closed-loop
deviation satisfies, for all $t\ge0$,
\begin{equation}
\label{eq:app-window}
\norm{v(t,\cdot)}_{L^1(\Om)}
\le
C_\sharp\int_{(t-T_\sharp)_+}^{t}\abs{e(s)}\,ds
+\mathbf 1_{\{t<T_\sharp\}}\,\norm{v_0}_{L^1(\Om)} .
\end{equation}
\end{lemma}

\begin{proof}
By \eqref{eq:app-Tsharp} every characteristic of \eqref{eq:app-v-eq} issued from
the inflow boundary reaches $l_m$ within time $T_\sharp$; hence for $t\ge T_\sharp$
the value $v(t,\cdot)$ depends only on data entering through $\{l=l_0\}$ on
$[t-T_\sharp,t]$ and on the interior source there, the contribution of $v_0$
having been swept out (Lemma~\ref{lem:main-sweep}). Along characteristics the
homogeneous part is a contraction ($c_E\ge0$), so integrating \eqref{eq:app-v-eq}
in $L^1$ as in \eqref{eq:app-w-diff},
\[
\frac{d}{dt}\norm{v(t)}_{L^1}
\le \underbrace{\abs{\Delta g(t,l_0)}x^*(l_0)}_{\text{inflow}}
+\underbrace{\int_\Om\big(\abs{\partial_l(\Delta g\,x^*)}+\abs{\Delta c\,x^*}\big)dl}_{\text{source}} .
\]
Since $\abs{\Delta g}\le C_M\abs{e}$, $\abs{\partial_l\Delta g}\le C_M\abs{e}$,
$\abs{\Delta c}\le C_M\abs{e}$ (first-order Lipschitz bound
\eqref{eq:model-ass-Elip}), and $x^*\in BV$ with $x^*(l_0)=p/g^*(l_0)$, the
right-hand side is $\le C_\sharp\abs{e(t)}$ with
$C_\sharp:=C_M\big(x^*(l_0)+\norm{x^*}_{L^1}+\abs{x^*}_{BV}\big)$. Integrating over
$[(t-T_\sharp)_+,t]$ and adding the transported initial datum when $t<T_\sharp$
gives \eqref{eq:app-window}.
\end{proof}

\begin{lemma}[Perturbation $BV$-window estimate]
\label{lem:app-v-bv-window}
Assume \eqref{eq:model-C2E}. There is a constant \par
\noindent $C_\flat=C_\flat\big(C_M,C_2,C_2',\Theta_M,g_M,\norm{x^*}_{W^{1,\infty}},\Om,T_\sharp\big)$
such that, for all $t\ge0$,
\begin{equation}
\label{eq:app-v-bv}
\abs{v(t,\cdot)}_{BV(\Om)}
\le
C_\flat\Big(\mathbf 1_{\{t<T_\sharp\}}\norm{v_0}_{BV}
+\sup_{s\in[(t-T_\sharp)_+,t]}\abs{e(s)}\Big).
\end{equation}
\end{lemma}

\begin{proof}
Under the standing assumptions the stationary profile
\eqref{eq:model-stationary-profile} is Lipschitz, $x^*\in W^{1,\infty}(\Om)$, with
$\norm{x^*}_{W^{1,\infty}}\le C(C_M,\Theta_M,g_M,\norm{p}_\infty)$. Write
\eqref{eq:app-v-eq} in nonconservative form,
\[
\partial_t v+g_E\partial_l v
=-\big(\partial_l g_E+c_E\big)v+h,
\qquad
h:=-\partial_l(\Delta g\,x^*)-\Delta c\,x^*,
\]
with $g_E:=g(E^*+e,\cdot)$, $c_E:=\mu(E^*+e,\cdot)+u$. Because
$\Delta g,\Delta c\in W^{1,\infty}(\Om)$ (differences of $W^{1,\infty}$
coefficients) and $x^*\in W^{1,\infty}(\Om)$, the source $h(t,\cdot)$ lies in
$BV(\Om)$, and by \eqref{eq:model-ass-Elip} and \eqref{eq:model-C2E},
$\displaystyle 
\norm{h(t,\cdot)}_{L^1}+\abs{h(t,\cdot)}_{BV}\le C\,\abs{e(t)}$.
Differentiating in $l$, the measure $\mu_v:=D_lv(t,\cdot)$ satisfies, in the sense
of measures along the characteristics $\dot L=g_E$ (speed $\ge g_M$, sweep-out
time $\le T_\sharp$),
\[
\partial_t\mu_v+g_E\partial_l\mu_v
=-\big(2\partial_l g_E+c_E\big)\mu_v
-\big(\partial_{ll}g_E+\partial_l c_E\big)v+D_lh .
\]
Here $\abs{2\partial_l g_E+c_E}\le 2C_M+\Theta_M$, and
$\abs{\partial_{ll}g_E}+\abs{\partial_l c_E}\le 2C_2+C_M$ by
\eqref{eq:model-ass-second} and $\norm{\partial_l u}_\infty\le C_2$; the last two
terms have total mass
\[
\int_\Om\abs{(\partial_{ll}g_E+\partial_l c_E)\,v}\,dl+\abs{D_lh}(\Om)
\le (2C_2+C_M)\norm{v(t)}_{L^1}+C\abs{e(t)} .
\]
The inflow boundary contributes $\abs{v(t,l_0)}\le C\abs{e(t)}$, by
\eqref{eq:app-v-bc} and $x^*(l_0)=p/g^*(l_0)$. Applying the measure-valued
$L^1$-stability estimate for $\mu_v$ as in \eqref{eq:app-w-diff} (with $c_E\ge0$
making the homogeneous flow a contraction) and Gr\"onwall over one sweep-out
window,
\[
\abs{v(t,\cdot)}_{BV}=\abs{\mu_v}(\Om)
\le e^{(2C_M+\Theta_M)T_\sharp}
\Big(\mathbf 1_{\{t<T_\sharp\}}\abs{v_0}_{BV}
+\!\!\int_{(t-T_\sharp)_+}^{t}\!\!\big[C\abs{e(s)}
+(2C_2+C_M)\norm{v(s)}_{L^1}\big]ds\Big).
\]
Bounding $\norm{v(s)}_{L^1}$ by Lemma~\ref{lem:app-window} and using
$\norm{v_0}_{L^1}\le(l_m-l_0)\norm{v_0}_{BV}$, the window integral is at most
$C\big(\sup_{[(t-T_\sharp)_+,t]}\abs{e}+\mathbf 1_{\{t<T_\sharp\}}\norm{v_0}_{BV}\big)$,
which gives \eqref{eq:app-v-bv}.
\end{proof}

\subsection{Backward adjoint and the quadratic bound}
\label{subsec:app-quadratic}

Fix $\omega\in(\omega_0,0)$. For the duality estimate we use, on each window
$[t-T_\sharp,t]$, the backward transport adjoint $\psi(\cdot,\cdot)$ defined by
\begin{equation}
\label{eq:app-backadj}
-\partial_s\psi-g^*\partial_l\psi+(\mu^*+u)\psi=0,\qquad s<t;\qquad
\psi(t,\cdot)=\chi,\qquad \psi(s,l_m)=0 .
\end{equation}
By the method of characteristics and $c_E\ge0$,
\begin{equation}
\label{eq:app-psi-bounds}
\norm{\psi}_{L^\infty}\le\norm{\chi}_{L^\infty},\qquad
\norm{\partial_l\psi}_{L^\infty}\le C_\psi
:=g_M^{-1}\big(\norm{\chi'}_{L^\infty}+(C_M+\Theta_M)g_M^{-1}\norm{\chi}_{L^\infty}\big)
e^{(C_M+\Theta_M)T_\sharp/g_M},
\end{equation}
and $\psi(s,\cdot)\equiv0$ for $s\le t-T_\sharp$ (backward nilpotency).

\begin{lemma}[Quadratic remainder]
\label{lem:app-quadratic}
Assume \eqref{eq:model-C2E}. There is a constant $C=C(C_M,C_2',\Theta_M,g_M,
\norm{\chi}_{W^{1,\infty}},\norm{x^*}_{BV},\Om,\omega)$ such that
\begin{equation}
\label{eq:app-quadratic}
\norm{\Gcal[e,v_0]}_\omega
\le
C\big(\norm{v_0}_{BV}+\norm{e}_\omega\big)^2 .
\end{equation}
\end{lemma}

\begin{proof}
Pair \eqref{eq:app-v-eq} with $\psi$ of \eqref{eq:app-backadj} over the window
$Q_t:=[(t-T_\sharp)_+,t]\times\Om$ and integrate by parts. By construction the
first-order parts of $\Delta g,\Delta c$ (namely $e\,\partial_Eg^*$,
$e\,\partial_E\mu^*$) together with the boundary term $-a_0e$ reproduce the linear
kernel $\kappa$ and free term $f$; what remains is
\[
\Gcal[e,v_0](t)
=\underbrace{-\!\int_{Q_t}\!\!R_g\,x^*\,\partial_l\psi}_{P_1}
\;\underbrace{-\!\int_{Q_t}\!\!R_\mu\,x^*\,\psi}_{P_2}
\;\underbrace{-\!\int_{Q_t}\!\!\Delta g\,\partial_l v\,\psi}_{P_3}
\;\underbrace{-\!\int_{Q_t}\!\!\Delta c\,v\,\psi}_{P_4}
\;-\;R_g(\cdot,l_0)x^*(l_0)\psi(\cdot,l_0),
\]
the boundary remainder being of the same quadratic type as $P_1$.

\emph{Terms $P_1,P_2$ (second-order coefficient).} By \eqref{eq:app-taylor},
$\abs{R_g}+\abs{R_\mu}\le\frac12 C_2'e^2$ and $\abs{\partial_l R_g}\le\frac12
C_2'e^2$, and by \eqref{eq:app-psi-bounds} $\norm{\psi}_\infty\le\norm{\chi}_\infty$,
$\norm{\partial_l\psi}_\infty\le C_\psi$. Hence, with $\norm{x^*}_{L^1}\le M_x(T)$,
\[
\abs{P_1}+\abs{P_2}
\le C_2'\big(C_\psi+\norm{\chi}_\infty\big)\norm{x^*}_{L^1}
\int_{(t-T_\sharp)_+}^{t}\abs{e(s)}^2\,ds .
\]

\emph{Terms $P_3,P_4$ (nonlinear products).} Using $\abs{\Delta g}\le C_M\abs{e}$,
$\abs{\Delta c}\le C_M\abs{e}$, $\norm{\psi}_\infty\le\norm{\chi}_\infty$ and the
window comparisons of Lemmas~\ref{lem:app-window} and
\ref{lem:app-v-bv-window} for $\norm{v(s)}_{L^1}$ and $\abs{v(s)}_{BV}$,
\[
\abs{P_3}+\abs{P_4}
\le C_M\norm{\chi}_\infty
\int_{(t-T_\sharp)_+}^{t}\abs{e(s)}
\Big(\abs{v(s)}_{BV}+\norm{v(s)}_{L^1}\Big)\,ds .
\]

\emph{Weighted assembly.} Multiply by $e^{-\omega t}$ and take the supremum over
$t$. For $P_1,P_2$, since $\abs{e(s)}\le\norm{e}_\omega e^{\omega s}$,
\[
e^{-\omega t}\!\int_{(t-T_\sharp)_+}^{t}\!\abs{e(s)}^2ds
\le \norm{e}_\omega^2\,e^{-\omega t}\!\int_{(t-T_\sharp)_+}^{t}\!e^{2\omega s}ds
\le \norm{e}_\omega^2\,T_\sharp e^{\abs{\omega}T_\sharp}.
\]
For $P_3,P_4$, Lemmas~\ref{lem:app-window} and \ref{lem:app-v-bv-window} give
$\norm{v(s)}_{L^1}+\abs{v(s)}_{BV}\le C_\sharp''\big(\sup_{[s-T_\sharp,s]}\abs{e}
+\mathbf 1_{\{s<T_\sharp\}}\norm{v_0}_{BV}\big)$, whence
\[
e^{-\omega t}\!\int_{(t-T_\sharp)_+}^{t}\!\abs{e(s)}\big(\norm{v(s)}_{L^1}
+\abs{v(s)}_{BV}\big)ds
\le C_\sharp''\,T_\sharp e^{\abs{\omega}T_\sharp}
\big(\norm{e}_\omega+\norm{v_0}_{BV}\big)\norm{e}_\omega .
\]
Collecting the four contributions,
\[
\norm{\Gcal[e,v_0]}_\omega
\le L_2\Big(\norm{e}_\omega^2
+\big(\norm{v_0}_{BV}+\norm{e}_\omega\big)\norm{e}_\omega\Big)
\le C\big(\norm{v_0}_{BV}+\norm{e}_\omega\big)^2,
\]
with $L_2:=3C_*T_\sharp e^{\abs{\omega}T_\sharp}$ and $C_*$ collecting
$C_2',C_M,C_\psi,\norm{\chi}_{W^{1,\infty}},\norm{x^*}_{BV},C_\sharp,C_\sharp',
C_\sharp''$. This is \eqref{eq:app-quadratic}, i.e.\ the hypothesis
\eqref{eq:main-quadratic-remainder}.
\end{proof}

Lemmas~\ref{lem:app-window}--\ref{lem:app-quadratic} supply the window comparison
and the quadratic remainder bound used in the proof of
Theorem~\ref{thm:main-nonlinear-stability}; under \eqref{eq:model-C2E} the
conclusion of that theorem therefore holds unconditionally.

\bibliography{reference}

\end{document}